\documentclass{amsart}

\usepackage{amssymb}
\usepackage{amscd,enumerate} 
\usepackage[utf8]{inputenc} 
\usepackage{color}

\makeatletter
\numberwithin{equation}{section}
\numberwithin{figure}{section}
\theoremstyle{plain}
\newtheorem{thm}{\protect\theoremname}[section]
\theoremstyle{definition}
\newtheorem{defn}[thm]{\protect\definitionname}
\theoremstyle{remark}
\newtheorem{rem}[thm]{\protect\remarkname}
\theoremstyle{plain}
\newtheorem{lem}[thm]{\protect\lemmaname}
\theoremstyle{definition}
\newtheorem{example}[thm]{\protect\examplename}
\theoremstyle{plain}
\newtheorem{prop}[thm]{\protect\propositionname}
\theoremstyle{plain}
\newtheorem{cor}[thm]{\protect\corollaryname}

\makeatother

\usepackage{babel}
\providecommand{\corollaryname}{Corollary}
\providecommand{\definitionname}{Definition}
\providecommand{\examplename}{Example}
\providecommand{\lemmaname}{Lemma}
\providecommand{\propositionname}{Proposition}
\providecommand{\remarkname}{Remark}
\providecommand{\theoremname}{Theorem}

\author[E. Garc\'\i a]{Esther Garc\'\i a}
\address{ Departamento de Matem\'{a}tica  Aplicada, Ciencia e Ingenier\'{\i}a de los Materiales y Tecnolog\'{\i}a Electr\'onica,
Universidad Rey Juan Carlos, 28933 M\'{o}s\-to\-les (Madrid), Spain}
\email{esther.garcia@urjc.es}

\author[M. G\'omez Lozano]{Miguel G\'omez Lozano}
\address{Departamento de \'Algebra, Geometr\'{\i}a y
Topolog\'{\i}a, Universidad de M\'alaga, 29071 M\'alaga, Spain}
\email{miggl@uma.es}

\author[R. Mu\~noz Alc\'azar] {Rub\'en Mu\~noz Alc\'azar}
\address{ Departamento de \'Algebra, Geometr\'{\i}a y
Topolog\'{\i}a, Universidad de M\'alaga, 29071 M\'alaga, Spain}
\email{rubenm.alcazar@uma.es}
\author[G. Vera de Salas]{Guillermo Vera de Salas}
\address{ Departamento de Matem\'{a}tica  Aplicada, Ciencia e Ingenier\'{\i}a de  Materiales y Tecnolog\'{\i}a Electr\'onica,
Universidad Rey Juan Carlos, 28933 M\'{o}s\-to\-les (Madrid), Spain}
\thanks{All authors were partially supported by the Junta de Andaluc\'{\i}a PPRO-FQM264-G-2023 (FQM264-G-FEDER) and by B42025-003: Ayudas para Proyectos Puente de la UMA}
\email{guillermo.vera@urjc.es}

\begin{document}
\title[Leibniz algebras and their connection to Jordan pair disystems]{Leibniz algebras and their connection to Jordan pair disystems}

\maketitle



\section*{Abstract}
In this paper we present a  generalization of Jordan pairs, called
Jordan pair disystem, which appear naturally as the wings of a   Leibniz algebra with a finite $\mathbb{Z}$-grading;  conversely,
we give the Tits-Kantor-Koecher construction for Jordan pair disystems, obtaining a Leibniz algebra with a short $\mathbb{Z}$-grading. This construction extends the construction given by Gubarev and Kolesnikov for Jordan dialgebras \cite{TKKDialgebras}.
We introduce homotopes for Jordan pair disystems at elements, extending the notion given in \cite{ASJ}, and we show that we obtain again a Jordan dialgebra positively solving the  open question posed in \cite[Remark 2]{ASJ}. Moreover, we introduce abelian inner ideals and their kernels  for  Leibniz algebras, and prove that the subquotient of a Leibniz algebra with respect to an abelian inner ideal is a Jordan pair disystem.
This notion of subquotient extends the construction of Jordan dialgebras at $Q$-Jordan elements given by Felipe and Vel\'asquez in \cite{QJAlg}.

\section{Introduction}

The concept of  Leibniz algebras was introduced by A. Bloh in 1965 \cite{Bloh} and rediscovered by Loday in 1993 as a non-anticommutative version of Lie algebras, see \cite{LodayLeibniz}.  As a generalization
of the Lie algebras, the skew-symmetric property is omitted and only
the property of Lie algebras that the right multiplication operators
are derivations is  preserved (in the context of (right) Leibniz algebras, this second property is called the (right) Leibniz identity).  More precisely, if $\phi$ is a ring of scalars, i.e., $\phi$ is an associative, commutative and unital ring, a
(right) Leibniz algebra $L$ is a module over $\phi$ such that
$
[[x,y],z]=[[x,z],y]+[x,[y,z]]
$
for every $x,y,z\in L$.

During the last 30 years, the theory of Leibniz algebras has been a subject of active research. In fact, many results from the theory of Lie algebras have been extended to the case of Leibniz algebras, although problems specific to this type of structure arise, see for example the monography by Ayupov, Omirov and Rakhimov \cite{AyuOmiRak}. For instance, J. L. Loday and T. Pirashvili proved in \cite{LodayPirashvili}
that  every Leibniz algebra over
a field can be embedded into an associative dialgebra, where an associative
dialgebra over a field is a vector space $D$ with two associative bilinear products
$
\dashv:D\times D\longrightarrow D$  and $\vdash:D\times D\longrightarrow D,
$
which
satisfy the identities:
\[
x\dashv(y\dashv z)=x\dashv(y\vdash z), \, (x\vdash y)\dashv z=x\vdash(y\dashv z)\text{ and } (x\vdash y)\vdash z=(x\dashv y)\vdash z
\]
for every $x,y,z\in D$. In this case, the vector space $D$ with
the product $[x,y]=x\dashv y-y\vdash x$ for every $x,y\in D$ is
a (right) Leibniz algebra denoted by $D^{(-)}$.

In 2008, R. Vel\'asquez and R. Felipe introduced
the notion of quasi-Jordan algebra, see \cite{QJAlg}. A (right) quasi-Jordan algebra
$J$  is
a $\phi$-module with a bilinear product $\bullet:J\times J\longrightarrow J$
such that
$$
x\bullet(y\bullet z)=x\bullet(z\bullet y)\hbox{ and }(y\bullet x)\bullet x^{2}=(y\bullet x^{2})\bullet x.
$$
As expected,
if $(D,\dashv,\vdash)$ is an associative dialgebra, then $D$ with the product
$x\bullet y=x\dashv y+y\vdash x$ for every $x,y\in D$ is a (right)
quasi-Jordan algebra denoted by $D^{(+)}$.
However,  in 2010  Bremner  proved that the adequate formalization of the class of algebras obtained in this way should involve an extra
identity: if $(D,\dashv,\vdash)$
is an associative dialgebra over a ring of scalars $\phi$, with $\frac{1}{2},\frac{1}{3}\in\phi$,
then the quasi-Jordan algebra $D^{(+)}$ also satisfies
\[
(x,y^{2},z)=2(x,y,z)\bullet y
\]
for every $x,y,z\in D^{(+)}$, where $(a,b,c)=(a\bullet b)\bullet c-a\bullet(b\bullet c)$, see \cite{BremnerQuasi}.
This identity was also obtained by Kolesnikov in \cite{K-identity}. Moreover, when applying the Kolesnikov-Pozhidaev algorithm to a linear Jordan algebra over a field of characteristic different from 2 and 3, Bremner, Felipe and S\'anchez-Ortega in \cite{BFSO} rediscovered this third identity as part of definition of a Jordan dialgebra (although the name dialgebra alludes to two products, the authors proved that these two products could be rewritten in terms of a single one).
 Therefore, from now on we will also assume this third identity and  a (right) Jordan dialgebra will be a $\phi$-module $J$ with a bilinear operation $\bullet$ such that
\[
x\bullet(y\bullet z)=x\bullet(z\bullet y),\ (y\bullet x)\bullet x^{2}=(y\bullet x^{2})\bullet x \text{ and }(x,y^{2},z)=2(x,y,z)\bullet y.
\]
The name of Jordan dialgebra and these three defining properties were also adopted by Gubarev and Kolesnikov in \cite{TKKDialgebras}; notice that Jordan dialgebras are also known as  K-B quasi-Jordan algebras, see for example \cite{KB} or \cite{ASJ}.

The aim of this paper is to relate Leibniz algebras and Jordan distructures.  Following the ideas of A. Fern\'andez L\'opez, E. Garc\'ia and
M. G\'omez Lozano in \cite{FGGJloc}, Vel\'asquez and Felipe showed in
\cite{QJAlg} and  \cite{OnK-B} how to attach a Jordan dialgebra $L_a$ to a (right) Leibniz algebra $L$ and an element $a\in L$ with  ${\rm ad}_a^3=0$.
E. Garc\'ia and M. G\'omez Lozano, together with A. Fern\'andez L\'opez and E.  Neher, gave in \cite{FGGN} another connection of Lie algebras and Jordan pairs by defining subquotients associated to abelian inner ideals of  Lie algebras.

In order to extend this construction to (right) Leibniz algebras, the first step will be to introduce an adequate notion of Jordan pair disystem.  Bremner, Felipe and S\'anchez-Ortega  already defined Jordan triple disystems in \cite{BFSO} by using the Kolesnikov-Pozhidaev algorithm, but Jordan pair disystems have not yet appeared in the literature. After  giving their definition, which involves 4 trilinear operations, we will prove that Jordan pairs disystems naturally appear as the wings of (right) Leibniz algebras with a finite $\mathbb{Z}$-grading and, conversely, from  every Jordan pair disystem
one can build a  (right) Leibniz algebra with a short $\mathbb{Z}$-grading following  the Tits-Kantor-Koecher construction, as was done in \cite{TKKDialgebras} for Jordan dialgebras by  Gubarev and Kolesnikov.

Once the notion of Jordan pair disystem is fixed, we will  use it to define homotopes of Jordan pair disystems at elements. Given a Jordan pair disystem $(V^+,V^-)$ and an element $a\in V^{-\sigma}$ we will define a Jordan dialgebra product in $V^\sigma$ via the element $a$. Homotopes of Jordan dialgebras at elements were defined by Alhefthi, Siddiqui and Jamjoom in 2023, but they were not able to show that the  product satisfied the third axiom of Jordan dialgebra and left it as an open question in  \cite[Remark 2]{ASJ}. We will answer this question in the affirmative by proving that homotopes of Jordan pair disystems, in particular, of Jordan dialgebras, are again Jordan dialgebras.

Finally, we introduce abelian inner ideals for (right) Leibniz algebras, their kernels and their associated subquotients, and will prove that the subquotient of a (right) Leibniz algebra associated to an abelian inner ideal is a Jordan pair disystem. When the abelian inner ideal comes from an element $a\in L$ with ${\rm ad}_a^3=0$ and the kernel of the element $a$ (in the sense of Felipe and Vel\'asquez) coincides with the kernel of the abelian inner ideal $B= [[L,a],a]$, the double of the Jordan dialgebra $L_a$ is isomorphic as a Jordan pair disystem to the subquotient associated to the abelian inner ideal, as it also happens in the Lie context  \cite{FGGN}.

\section{Preliminaries}

From now on, $\phi$ will denote a ring of scalars, i.e., an associative, commutative and unital ring. We will also suppose that $\frac 12,\frac 13\in \phi$.
\begin{defn}
Let $L$
be a module over $\phi$. We say that $L$ is a\textbf{ (right) Leibniz
algebra} if there exists a bilinear product $[\text{ },\text{ }]:L\times L\longrightarrow L$
which satisfies the (right) Leibniz identity
\[
[[x,y],z]=[[x,z],y]+[x,[y,z]]
\]
for every $x,y,z\in L$.
\end{defn}

\begin{rem}
By the (right) Leibniz identity we have $[x,[y,y]]=0$ for every
$x,y\in L$, and this implies right anticommutativity, i.e.,
\[
[x,[y,z]]=-[x,[z,y]],
\]
for every $x,y,z\in L$, because $[x,[y+z,y+z]]=[x,[y,z]]+[x,[z,y]] = 0$ for every $x,y,z\in L$.
\end{rem}

 We denote
by $L^{ann}$ the submodule of $L$ spanned by elements of the form
$[x,x]$, with $x\in L$. Note that for every $x,y\in L$
\[
[x+y,x+y]=[x,x]+[x,y]+[y,y]+[y,x],
\]
so $[x,y]+[y,x]\in L^{ann}$. It is easy to check that $L^{ann}$ is a two-sided ideal of $L$, and the quotient of the (right) Leibniz algebra $L$ by the ideal $L^{ann}$ gives
a Lie algebra denoted by $\bar L$. Moreover, the ideal $L^{ann}$
is the smallest two-sided ideal of $L$ such that $L/L^{ann}$ is
a Lie algebra.


The $\phi$-module $\hat{L}=L\oplus L/L^{ann}$ with product
$$
\langle a+\bar x, b+\bar y\rangle=[a,y]-[b,x]+\overline{[x,y]}
$$
becomes a Lie algebra. Observe that $\hat L$ is the split null extension of the Lie algebra $L/L^{ann}$ by the  $(L/L^{ann}, L/L^{ann})$-bimodule $L$, see \cite{Poz}. It is well-known that $L$ is a zero-square ideal of $\hat L$; notice that the product of the Leibniz algebra $L$ can be seen, under this construction, as $[a,y]=\langle a,\bar y\rangle$ for every $a,y\in L$.

\begin{defn}\label{QJordan}
Let $L$ be a (right) Leibniz algebra over a ring of scalars $\phi$.
Given $x\in L$ we define the \textbf{adjoint operator of} $x$ as
the linear map $\text{ad}_{x}:L\longrightarrow L$ given by $\text{ad}_{x}(y)=[y,x]$
for every $y\in L$. An element $x\in L$ is said to be an \textbf{ad-nilpotent element
of index} $n\in\mathbb{N}$ if $\text{ad}_{x}^{n}(L)=0$ and $\text{ad}_{x}^{n-1}(L)\neq0$. In particular, an element $a\in L$ is said to be a \textbf{$Q$-Jordan
element} if $a$ is ad-nilpotent of index at most $3$, see \cite[Definition 27]{QJAlg}.

We will denote by $\text{End}(L)=\text{End}^{l}(L)$ the associative algebra of $\phi$-linear maps $f:L\to L$ acting on the left, i.e., $f(x)$, for every $x\in L$; we will use the notation $\text{End}^{r}(L)$ to refer to $\phi$-linear maps $g:L\to L$ acting on the right, i.e., $(x)g$, for every $x\in L$.

Throughout this paper we will denote by capital letters the adjoint maps acting on the left, i.e., $X\in\text{End}^{r}(L)$ is given by  $(y)X:=[y,x]={\rm ad}_x(y)$ for every $x,y\in L$. By the (right) Leibniz identity, for every $x,y,z\in L$, ${\rm ad}_{[x,y]}(z)=[z,[x,y]]=[[z,x],y] -[[z,y],x]=(z)(XY-YX)$; if for capital letters the bracket means  the usual antisymmetrization commutator, we have  $(z)[X,Y]=(z)(XY-YX)$ for every $z\in L$, i.e., $[X,Y]=XY-YX$.

For any $Q$-Jordan element $a\in L$, $[[[x,a],a],a]=0$, so $0={\rm ad}_{[[[x,a],a],a]}(z)=(z)[[[X,A],A],A]=(z)(XA^3-3AXA^2+3A^2XA-A^3X)$ for every $z\in L$, and therefore $$AXA^2=A^2XA\quad \hbox{ and, in particular, }\quad  A^2XA^2=0.\eqno{(1)}$$
\end{defn}

\begin{defn}
A \textbf{(right) Jordan dialgebra}
is a $\phi$-module $J$
equipped with a bilinear product $\bullet:J\times J\longrightarrow J$
verifying the following axioms:
\begin{enumerate}
\item[(JD1)] right commutativity: $x\bullet(y\bullet z)=x\bullet(z\bullet y)$,
\item[(JD2)] right Jordan identity: $(y\bullet x)\bullet x^{2}=(y\bullet x^{2})\bullet x$,
\item[(JD3)] Osborn identity: $(x\bullet y^{2})\bullet z-x\bullet(y^{2}\bullet z)=2\left(((x\bullet y)\bullet z)\bullet y-(x\bullet(y\bullet z))\bullet y\right),$
\end{enumerate}
for every $x,y,z\in J$, where $x^{2}=x\bullet x$.
\end{defn}

Analogously to the case of (right) Leibniz algebras, if $J$ is a Jordan dialgebra
we can consider  the submodule $J^{ann}$ of $J$ spanned by elements of the form
$x\bullet y-y\bullet x$, with $x,y\in J$, and $Z^{r}(J)=\{z\in L:x\bullet z=0 \text{ for all } x\in J\}$. Then $J^{ann}$ and $Z^{r}(J)$ are two-sided ideals of $J$ satisfying $J^{ann}\subseteq Z^{r}(J)$ and $Z^{r}(J)\bullet J\subseteq J^{ann}$. Moreover, the quotient algebra $\bar J:=J/J^{ann}$ is a Jordan algebra, and  $J^{ann}$ is the smallest two-sided ideal of $J$ such that $J/J^{ann}$ is a Jordan algebra.

Also the split null extension of the Jordan algebra $J/J^{ann}$ by the  $(J/J^{ann},J/J^{ann})$-bimodule $J$, $\hat J=J\oplus J/J^{ann}$, is a Jordan algebra with product
$$
(a+\bar x) \circ (b+\bar y)=a\bullet y+b\bullet x+\overline{x\bullet y},
$$
and the product of the Jordan dialgebra $J$ can be recovered from the above structure as
$$
a\bullet y=a \circ \bar y
$$
for every $a,y\in J$.

\begin{rem}\label{localRaules}
  In \cite[Theorem 36 and 37]{QJAlg} Vel\'asquez and Felipe showed how to attach a Jordan dialgebra to a (right) Leibniz algebra via a  $Q$-Jordan element:
if  $L$ is a (right) Leibniz algebra and
$a$ is a $Q$-Jordan element of $L$, then $L$ with the new product
defined by
$
x\circ y=\frac{1}{2}[x,[y,a]]
$
is a nonassociative algebra, denoted by $L^{(a)}$, such that\textup{
$
\text{ker}_{L}\{a\}:=\{z\in L:[[z,a],a]=0\}
$
}is an ideal of $L^{(a)}$. Moreover, \textup{$L_{a}:=L/\text{ker}_{L}\{a\}$}
is a Jordan dialgebra called the \textbf{Jordan dialgebra of $L$ at $a$}.
\end{rem}

\section{Jordan pair disystems}

In this section we define the concept of right Jordan pair disystem and we show that  they  naturally appear when considering the wings of (right) Leibniz algebras with a finite $\mathbb{Z}$-grading.
Let us first recall the notion of (linear) Jordan pair, due to K. Meyberg and  O. Loos, which is valid when $\frac 12,\frac 13\in \phi$.
\begin{defn}\label{JordanPair}
We say that a pair of $\phi$-modules $(V^+,V^-)$ is a (linear) Jordan pair  if there are two-trilinear maps $\{\text{ },\text{ },\text{ }\}^\sigma:V^{\sigma}\times V^{-\sigma}\times V^{\sigma}\longrightarrow V^{\sigma}$, $\sigma=\pm$, such that
for every $x,z,v\in V^{\sigma}$ and every $y,u,w\in V^{-\sigma}$,
\begin{enumerate}
  \item[(JP1)] $\{x,y,z\}^\sigma=\{z,y,x\}^\sigma$
  \item[(JP2)] $\{\{x,y,z\}^\sigma,u,v\}^\sigma = \{\{x,u,v\}^\sigma,y,z\}^\sigma - \{x, \{y,v,u\}^{-\sigma}, z\}^\sigma + \{x,y,\{z,u,v\}^\sigma\}^\sigma$.
\end{enumerate}
\end{defn}

\begin{defn}
Let $V=(V^{+},V^{-})$ be a pair of modules over $\phi$. We say
that $V$ is a (linear) \textbf{(right) Jordan pair disystem} if there
are four trilinear maps
\[
\{\text{ },\text{ },\text{ }\}_i^{\sigma}:V^{\sigma}\times V^{-\sigma}\times V^{\sigma}\longrightarrow V^{\sigma},
\]
$i=1,2$,  $\sigma = \pm$, such that for every $x,z,v\in V^{\sigma}$ and every $y,u,w\in V^{-\sigma}$, $\sigma=\pm$,
{\small \begin{enumerate}[{(JPD}1)]
\item $\{x,y,z\}_2^\sigma = \{z,y,x\}_2^\sigma,$
\item $\{w,\{x,y,z\}_1^{\sigma},u\}_1^{-\sigma} = \{w,\{x,y,z\}_2^{\sigma},u\}_1^{-\sigma}$,
\item $\{v,u,\{x,y,z\}_1^\sigma\}_1^\sigma = \{v,u,\{x,y,z\}_2^\sigma\}_1^\sigma,$
\item[(JPD4)] $\{\{x,y,z\}_1^\sigma,u,v\}_2^\sigma = \{\{x,y,z\}_2^\sigma,u,v\}_2^\sigma$,
\item[(JPD5)] $\{\{x,y,z\}_1^\sigma,u,v\}_1^\sigma = \{\{x,u,v\}_1^\sigma,y,z\}_1^\sigma - \{x, \{y,v,u\}_1^{-\sigma}, z\}_1^\sigma + \{x,y,\{z,u,v\}_1^\sigma\}_1^\sigma$,
\item[(JPD6)] $\{\{x,y,z\}_2^\sigma,u,v\}_1^\sigma = \{\{x,u,v\}_1^\sigma,y,z\}_2^\sigma - \{x, \{y,v,u\}_1^{-\sigma}, z\}_2^\sigma + \{x,y,\{z,u,v\}_1^\sigma\}_2^\sigma$,
\item[(JPD7)] $\{v,u,\{x,y,z\}_1^\sigma\}_1^\sigma = \{\{v,u,x\}_1^\sigma,y,z\}_1^\sigma - \{x, \{u,v,y\}_2^{-\sigma}, z\}_2^\sigma + \{\{v,u,z\}_1^\sigma,y,x\}_1^\sigma$,
\item[(JPD8)] $\{v,u,\{x,y,z\}_1^\sigma\}_2^\sigma = \{ \{v,u,x\}_2^\sigma,y,z\}_1^\sigma - \{x, \{u,v,y\}_1^{-\sigma},z\}_2^\sigma + \{\{v,u,z\}_2^\sigma,y,x\}_1^\sigma$.
\end{enumerate}}
\noindent In the sequel we will omit the superindices $\sigma$ in the products $\{\text{ },\text{ },\text{ }\}_1$ and $\{\text{ },\text{ },\text{ }\}_2$ when they can be easily deduced from the context.
\end{defn}

\begin{rem}\label{doubleJ}
In \cite[Definition 4.3]{BFSO} the authors introduced the notion of Jordan triple disystem $(T, \{\text{ },\text{ },\text{ }\}_1, \{\text{ },\text{ },\text{ }\}_2)$ by applying the KP-algorithm to a Jordan triple system.
They also showed in \cite[Theorem 7.3, Propositions 7.7 and 7.10]{BFSO} how to obtain a Jordan triple disystem from a Jordan dialgebra over a ring of scalars by defining
  \begin{align*}
        \{x,y,z\}_1&:=x\bullet(y\bullet z)-(x\bullet z)\bullet y+(x\bullet y)\bullet z,
        \\
        \{x,y,z\}_2& :=  (y \bullet z) \bullet x + (y \bullet x) \bullet z - y \bullet (x \bullet z).
    \end{align*}
In particular, this can be done starting with a  Jordan dialgebra arising from an associative dialgebra (see \cite{QJAlg}).
It is straightforward  to check that the double $(T,T)$ of a Jordan triple disystem is a Jordan pair disystem.

In a similar way, if one applies the  KP-algorithm to a (linear) Jordan pair and follows the arguments of \cite[Theorem 7.3, Propositions 7.7 and 7.10]{BFSO}  one arrives to the definition of Jordan pair disystem given above.
\end{rem}

\begin{defn}
    Let $V=(V^{+},V^{-})$ be a Jordan pair disystem. A pair of $\phi$-submodules
    $I=(I^{+},I^{-})$, $I^{\sigma}\subseteq V^{\sigma}$, where $\sigma=\pm$,
    is an \textbf{ideal of} $V$ if
    \[
        \{V^{\sigma},V^{-\sigma},I^{\sigma}\}_i+\{V^{\sigma},I^{-\sigma},V^{\sigma}\}_i+\{I^{\sigma},V^{-\sigma},V^{\sigma}\}_i \subseteq I^{\sigma}
    \]
    for  $i=1,2$ and $\sigma=\pm$.
    \end{defn}

\begin{prop}
    Let $V=(V^{+},V^{-})$ be a Jordan pair disystem  and let $V^{ann}=(\left(V^{ann}\right)^{+},\left(V^{ann}\right)^{-})$, where $\left(V^{ann}\right)^{\sigma}$ is the submodule of $V^{\sigma}$ generated by the elements of the form $$\{x,y,z\}_1-\{x,y,z\}_2,$$
    $x,z\in V^{\sigma}$, $y\in V^{-\sigma}$,  $\sigma=\pm$. Then
    $V^{ann}$ is an ideal of $V$ and $V/V^{ann}$  is a Jordan pair over $\phi$.
\end{prop}

\begin{proof}
Let us see that $V^{ann}$ is an ideal of $V$: take $\sigma\in \{+,-\}$ and let $x,z,v \in V^\sigma$ and $y, u \in V^{-\sigma}$.
\begin{enumerate}
    \item By (JPD3) we have that $\{0\} = \{V^\sigma, V^{-\sigma}, (V^{ann})^\sigma\}_1 \subseteq (V^{ann})^\sigma$.

    \item On the other hand, by (JPD1) and (JPD4), we can prove $$\{0\} = \{V^\sigma, V^{-\sigma}, (V^{ann})^\sigma\}_2 \subseteq (V^{ann})^\sigma.$$

    \item By (JPD2) we have that $0 = \{V^\sigma, (V^{ann})^{-\sigma}, V^{\sigma}\}_1 \subseteq (V^{ann})^\sigma$.

    \item  Using (JPD5) and (JPD6):
    \begin{align*}
        \{\{x,y,z\}_1,& u, v\}_1 - \{\{x,y,z\}_2, u,v\}_1 = \\
        &= \{\{x,u,v\}_1,y,z\}_1 - \{x, \{y,v,u\}_1,z\}_1 + \{x,y,\{z,u,v\}_1\}_1 \\
        & - \{\{x,u,v\}_1,y,z\}_2 + \{x, \{y,v,u\}_1,z\}_2 - \{x,y,\{z,u,v\}_1\}_2 \\
        &= \{\{x,u,v\}_1,y,z\}_1- \{\{x,u,v\}_1,y,z\}_2 \\
        &-(  \{x, \{y,v,u\}_1,z\}_1 - \{x, \{y,v,u\}_1,z\}_2) \\
        &+ \{x,y,\{z,u,v\}_1\}_1 - \{x,y,\{z,u,v\}_1\}_2 \in (V^{ann})^\sigma.
    \end{align*}

    \item Using (JPD7), (JPD8) and the previous item:
    \begin{align*}
        \{x,\{u,v,\,&y\}_1, z\}_2 - \{x, \{u,v,y\}_2, z\}_2 = \\
        &= \{\{v,u,x\}_2,y,z\}_1 - \{v,u,\{x,y,z\}_1\}_2 + \{\{v,u,z\}_2,y,x\}_1 \\
        &-\{\{v,u,x\}_1, y, z\}_1 + \{v,u,\{x,y,z\}_1\}_1 - \{\{v,u,z\}_1,y,x\}_1 \\
        &= -(\{\{v,u,x\}_1, y, z\}_1 - \{\{v,u,x\}_2,y,z\}_1)\\
        &+ \{v,u,\{x,y,z\}_1\}_1- \{v,u,\{x,y,z\}_1\}_2 \\
        &- (\{\{v,u,z\}_1,y,x\}_1 - \{\{v,u,z\}_2,y,x\}_1) \in (V^{ann})^\sigma.
    \end{align*}

    \item By (JPD4) we have that $\{0\} = \{(V^{ann})^\sigma, V^{-\sigma}, V^\sigma\}_2 \subseteq (V^{ann})^\sigma.$
\end{enumerate}
 Moreover, since in the quotient
$
\overline{\{x,y,z\}}_1 = \overline{\{x,y,z\}}_2
$ we have that $V/V^{ann}$  is a Jordan pair.

Notice that the eight axioms (JPD1)-(JPD8) are actively used in this proof, showing that none of them are redundant.
\end{proof}

\begin{rem}\label{rem:jp}
    Notice that $V^{ann}$ is the smallest ideal of $V$ such that the quotient $\bar V=V/V^{ann}$ is a Jordan pair. The split null extension of the Jordan pair $V/V^{ann}$ by the $(V/V^{ann},V/V^{ann})$-Jordan module $V$ is $\hat V=V\oplus V/V^{ann}$ with products
    $$
    \langle a+\bar x,\, b+\bar y,\, c+\bar z\rangle^{\sigma}=\{a,y,z\}_1^{\sigma}+\{x,b,z\}_2^{\sigma}+\{c,y,x\}_1^{\sigma}+\overline{\{x,y,z\}_1^{\sigma}}, \ \sigma=\pm
    $$
    and, as before, the products $\{\ \, , \, , \,\}_1$ and $\{\ \, , \, , \, \}_2$ of the Jordan pair disystem can be recovered from the above structure.
\end{rem}

In the following result we will show that, as it occurs with Jordan pairs and $\mathbb{Z}$-graded Lie algebras, Jordan pair disystems appear naturally when considering the wings of a (right) Leibniz algebra with a finite $\mathbb{Z}$-grading. Recall that  $L$ has a finite  $\mathbb{Z}$-grading if $L=L_{-n}\oplus...\oplus L_{0}\oplus...\oplus L_{n}$ with $[L_i,L_j]\subset L_{i+j}$ when $|i+j|\le n$ and $[L_i,L_j]=\{0\}$ otherwise.

\begin{thm}\label{AssociatedPair}
Let $L=L_{-n}\oplus...\oplus L_{0}\oplus...\oplus L_{n}$ be a (right) Leibniz algebra with a finite $\mathbb{Z}$-grading. Then $V=(L_{n},L_{-n})$,
with the triple products given by
$$\{x,y,z\}_1=[x,[y,z]], \quad \{x,y,z\}_2 = -[[y,x],z],$$
for every $x,z\in L_{\sigma n}$, $y\in L_{-\sigma n}$, $\sigma=\pm$, is a Jordan pair disystem.
\end{thm}

\begin{proof}
Let $x,z,v\in L_{\sigma n}$ and $y,u,w\in L_{-\sigma n}$, with $\sigma=\pm$.
Then it satisfies:
\begin{enumerate}
\item [(JPD1)]
   $$\{x,y,z\}_2 = -[[y,x],z] = -[[y,z],x] -[y,[x,z]] = -[[y,z],x] = \{z,y,x\}_2.$$
\item [(JPD2)]
\begin{align*}
    \{w,\{x,y,z\}_1,v\}_1&= \{w, [x,[y,z]], v \}_1 = [w,[[x,[y,z]],v]] = [w,[[[x,y],z],v]]\\
    &=-[w,[-[[y,x],z],v]] =\{w, \{x,y,z\}_2, v\}_1.
\end{align*}
\item [(JPD3)]
\begin{align*}
    \{v,u,\{x,y,z\}_1\}_1&=    [v,[u,[x,[y,z]]]] = -[v,[u,[[y,z],x]]]=\{v,u,\{z,y,x\}_2\}_1 \\&= \{v, u, \{x,y,z\}_2\}_1.
\end{align*}
\item [(JPD4)]
\begin{align*}
\{\{x,y,z\}_1, u, v\}_2 &=   -[[u,[x,[y,z]]],v] = [[u,[[y,z],x]],v]= \{\{x,y,z\}_2,u,v\}_2. \\
\end{align*}
\item [(JPD5)]
\begin{align*}
    \{\{x,&y,z\}_1, u,v \}_1 =  [[x,[y,z]],[u, v]]\\
    &=[[x,[u, v]],[y,z]]+[x,[[y,[u, v]],z]]+[x,[y,[z,[u, v]]]] \\
    &= \{\{x,u,v\}_1,y,z\}_1 -[x,[[y,[v, u]],z]]+\{x,y,\{z,u,v\}_1\}_1 \\
    & = \{\{x,u,v\}_1,y,z\}_1 - \{x, \{y,v,u\}_1,z\}_1+\{x,y,\{z,u,v\}_1\}_1.
\end{align*}

\item [(JPD6)]
\begin{align*}
    \{\{x,&y,z\}_2, u,v \}_1 =  -[[[y,x],z],[u,v]]\\
    &=-[[y,[u,v]],x],z]-[y,[x,[u,v]]],z]-[[y,x],[z,[u,v]]]\\
    & = [[y,[v,u]],x],z]+\{\{x,u,v\}_1, y, z\}_2+ \{x,y,\{z,u,v\}_1\}_2 \\
    &=  - \{x, \{y,v,u\}_1, z\}_2+\{\{x,u,v\}_1, y, z\}_2 + \{x,y,\{z,u,v\}_1\}_2.
\end{align*}

\item [(JPD7)]
\begin{align*}
    \{v,&u,\{x,y,z\}_1\}_1 =  [v,[u,[x,[y,z]]]]=[[v,u],[x,[y,z]]]\\
    &=([v,u])X[Y,Z]-([v,u])[Y,Z]X \\
    &=([v,u])X[Y,Z]-([v,u])YZX+([v,u])ZYX\\
    &=[[[v,u],x]],[y,z]]-([v,u])YZX+([v,u])Z[Y,X]\\
    &=[[v,[u,x]],[y,z]]-([v,u])YXZ+[[[v,u],z]],[y,x]]\\
    &=\{\{v,u,x\}_1,y,z\}_1-[[[v,u],y],x],z]+[[v,[u,z]],[y,x]] \\&=
   \{\{v,u,x\}_1,y,z\}_1-\{x,\{u,v,y\}_2,z\}_2+\{\{v,u,z\}_1,y,x\}_1
\end{align*}

\item [(JPD8)]
\begin{align*}
    \{v,&u,\{x,y,z\}_1\}_2 =-[[u,v],[x,[y,z]]]\\
    &=-[[[u,v],x],[y,z]]+[[[[u,v],y],z],x]-[[[[u,v],z],y],x]\\&=-[[[u,v],x],[y,z]]+[[[[u,v],y],z],x]-[[[u,v],z],[y,x]]\\
    &=\{ \{v,u,x\}_2,y,z\}_1 - \{z, \{u,v,y\}_1,x\}_2 + \{\{v,u,z\}_2,y,x\}_1
\end{align*}

\end{enumerate}
Therefore, $V=(L_{n},L_{-n})$ is a Jordan pair disystem.
\end{proof}

\begin{rem}\label{rem:wingsLeibniz}
    Let $L=L_{-n}\oplus...\oplus L_{0}\oplus...\oplus L_{n}$ be a (right) Leibniz algebra with a finite $\mathbb{Z}$-grading.  Notice that $L^{ann}$ is $\mathbb{Z}$-graded ideal of $L$, and therefore both $\bar L=L/L^{ann}$ and the split null extension $\hat L=L\oplus L/L^{ann}$ are also  $\mathbb{Z}$-graded Lie algebras with the same support. Hence, the wings $(\hat L_{-n}, \hat L_n)$ of $\hat L$ have the structure of Jordan pair with the products
    $$
        \langle a+ \bar x, b+ \bar y, c+ \bar z \rangle^{\sigma}=\langle  a+ \bar x, \langle b+ \bar y ,  c+ \bar z \rangle\rangle
    $$
    for every $a+ \bar x, c+ \bar z\in \hat L_{\sigma n}$ and every  $b+ \bar y\in \hat L_{-\sigma n}$, $\sigma=\pm$.
    The products $\{\ \, , \, ,\,  \}_1$ and $\{\ \, , \, ,\, \}_2$ defined in Theorem \ref{AssociatedPair}  can be recovered from the above structure:

    \begin{align*}
    \{a,y,z\}_1&=[a,[y,z]]=\langle a, \overline{[y,z]}\rangle =\langle  a, \langle\bar y,\, \bar z\rangle\rangle=\langle a,\, \bar y,\,\bar z  \rangle^\sigma,\\
    \{x,b,z\}_2 &=-[[b,x],z]=-[[b,z],x] \hbox{ (by the $\mathbb{Z}$-grading)}= \langle \bar x, [b,z]\rangle\\
    & =\langle  \bar x, \langle b, \bar z \rangle\rangle=\langle\bar x, b, \bar z \rangle^\sigma.
    \end{align*}

    These relations  provide an alternative way of proving that $(L_{-n},L_n)$ satisfies the identities of Jordan pair disystem. For example, if  $x,z,v\in L_{\sigma n}$ and $y,u,w\in L_{-\sigma n}$,  $\sigma=\pm$, then
    \begin{enumerate}
      \item [(JPD3)]
        \begin{align*}
            \{v,u,\{x,y,z\}_1\}_1&= \langle v, \bar u, \overline{\{x,y,z\}}_1 \rangle^\sigma=\langle v, \bar u, \overline{\{x,y,z\}}_2 \rangle^\sigma\\
            &= \{v, u, \{x,y,z\}_2\}_1
        \end{align*}
        where we have used that $\overline{\{x,y,z\}}_1=\overline{\{x,y,z\}}_2$.

    \item [(JPD6)]
    \begin{align*}
        \{\{x,&y,z\}_2, u,v \}_1 =  \langle \langle \bar x, y,  \bar z\rangle , \bar u, \bar  v\rangle^\sigma = \langle \langle  \bar  x,  \bar u, \bar v\rangle ,y ,\bar z\rangle^\sigma
        \\
    &-\langle \bar x,\langle y, \bar v, \bar u\rangle , \bar z\rangle^\sigma + \langle \bar x,y,\langle \bar z, \bar u, \bar v\rangle \rangle^\sigma
        \\
    &=\langle  \overline{\{x,u,v\}}_1, y, \bar z\rangle^\sigma - \{x,\{y,v,u\}_1,z\}_2 + \langle  \bar x, y,  \overline{\{z,u,v\}}_1\rangle^\sigma
        \\
    &=  \{\{x,u,v\}_1, y, z\}_2- \{x, \{y,v,u\}_1, z\}_2 + \{x,y,\{z,u,v\}_1\}_2
    \end{align*}
    where we have used that  $(\hat L_{-n},\hat L_n)$ satisfies (JP2) of Definition \ref{JordanPair}.
    \end{enumerate}
\end{rem}

\begin{example}
  The following example of a Leibniz algebra with finite $\mathbb{Z}$-grading is due to Liu and can be found in \cite[Example 31]{OnK-B}:
Let $\phi$ be a field and consider the $\phi$-span of $\{e,f,h,u,v,w\}$ with products
$$  \begin{array}{llll}
    \,[h,e]=2e  & [h,f]=-2f  & [e,h]=-2e & [e,f]=h  \\
    \,[f,h]=2f & [f,e]=-h  & [u,h]=-2u & [u,f]=-v \\
    \,[v,e]=-2u & [v,f]=-w & [w,h]=2w & [w,e]=-2v
     \end{array}
$$
and the rest of the products are zero. It is easy to check that $L^{ann}={\rm span}_{\phi}\{u,v,w\}$. Moreover,
 $L$  has the following 3-grading
$$
L_{-1}={\rm span}_{\phi}\{f,w\}, \qquad L_0={\rm span}_{\phi}\{h,v\},\qquad L_{1}={\rm span}_{\phi}\{e,u\}
$$
By Theorem \ref{AssociatedPair} $(L_{-1},L_{1})$ is a Jordan pair disystem. Notice that $$(L_{-1},L_{1})^{ann}=(L^{ann}\cap L_{-1}, L^{ann}\cap L_{1})=({\rm span}_{\phi}\{w\}, {\rm span}_{\phi}\{u\})$$ because
$$
\{w,e,f\}_1=[w,[e,f]]=[w, h ]=2w\quad\quad \{w,e,f\}_2=-[[e,w],f]=0.
$$
$$
\{u,f,e\}_1=2u \qquad \qquad \{u,f,e\}_2=0.
$$
\end{example}

\section{The Tits-Kantor-Koecher construction}

In this section we will show how to recover a Leibniz algebra with a short $\mathbb{Z}$-grading from a Jordan pair disystem.

\begin{defn}
    Let $V = (V^+, V^-)$ be a (right) Jordan pair disystem. Associated to $x\in V^\sigma$ and $y\in V^{-\sigma}$, where $\sigma=\pm$, we define the element $$[x,y]\in \text{End}^r(V^+)\times \text{End}^r(V^-)\times\text{End}^{l}(V^+)\times\text{End}^{l}(V^-)$$
   where for $z\in V^\sigma$ and $u\in V^{-\sigma}$ we have
\begin{align*}
 [x,y](z):=\{x,y,z\}_1,&\qquad  (z)[x,y]:=-\{z,y,x\}_1,\\
  [x,y](u):=-\{y,x,u\}_2,& \qquad  (u)[x,y]:=\{u,x,y\}_1.
\end{align*}
 Let $D(V)$ be the $\phi$-module generated by all $[x,y]$ where  $x \in V^\sigma$, $y \in V^{-\sigma}$,  $\sigma=\pm$.
\end{defn}

\begin{thm}\label{thm:TKK}
Let $V = (V^+, V^-)$ be a (right) Jordan pair disystem. Consider  the $\phi$-module
$$
    \mathrm{TKK}(V):= V^- \oplus D(V) \oplus V^+
$$
with product
$$[x+G+y,z+H+u]:=((x)H+G(z))+([x,u]+[G,H]+[y,z])+(G(u)+(y)H)$$
 for every $x,z\in V^\sigma$ and $y,u\in V^{-\sigma}$, $G,H\in D(V)$
where $[G,H]$ is extended by linearity from
\begin{align*}&[[x,y],[z,u]]:=-[\{x,u,z\}_1,y]+[x,\{y,z,u\}_1],\\
&[[x,y],[u,z]]:=-[[x,y],[z,u]].
\end{align*}
Then $\mathrm{TKK}(V)$
 is a $\mathbb{Z}$-graded (right) Leibniz algebra with $L_1 = V^+$,  $L_0 = D(V)$, $L_{-1} = V^-$ and $L_k = \{0\}$ for any $|k| > 1$.  In particular, $D(V)$ is a (right) Leibniz algebra.
\end{thm}
\begin{proof}
Let $\hat V  = V \oplus V/V^{ann} = ( \hat V^+, \hat V^-)$ be the Jordan pair with products $\langle\ ,\ ,\  \rangle^{\sigma}$, $\sigma=\pm$, defined in Remark \ref{rem:jp}. Let $\hat L = \hat L _{-1} \oplus \hat L_0 \oplus \hat L_1$ be the TKK Lie algebra of $\hat V$ with product denoted by $\langle \ , \ \rangle$.

Let us prove that, as in Remark \ref{rem:wingsLeibniz}, we can recover the product of $\mathrm{TKK}(V)$ from that of the Lie algebra $\hat L$: for every $x,z,v\in V^\sigma$ and $y,u,w\in V^{-\sigma}$, $\sigma=\pm$, we have
       \begin{align*}
      [x,z]&=\langle x , \overline z \rangle=0,\\
      [z,[x,y]]&=(z)[x,y]=-\{z,y,x\}_1=-\langle z, \overline y,\overline x\rangle^\sigma=-\langle z,\langle \overline
      y,\overline x\rangle\rangle=\langle z,\langle \overline x,\overline y\rangle\rangle,\\
      [u,[x,y]]&=(u)[x,y]=\{u,x,y\}_1=\langle u, \overline x,\overline y\rangle^{-\sigma}=\langle u,\langle \overline x,\overline y\rangle\rangle,\\
      [[x,y],z]&=[x,y](z)=\{x,y,z\}_1=\langle x, \overline y,\overline z\rangle^\sigma=\langle x,\langle \overline y,\overline z\rangle\rangle=\langle \langle x,  \overline y\rangle,\overline z\rangle,\\
      [[x,y],u]&=[x,y](u)=-\{y,x,u\}_2=-\langle \overline y,  x,\overline u\rangle^{-\sigma}=-\langle \langle\overline y,  x\rangle,\overline u\rangle=\langle \langle x,  \overline y\rangle,\overline u\rangle,\\
      [[x,y],[z,u]]&=-[\{x,u,z\}_1,y]+[x,\{y,z,u\}_1]=-\langle \{x,u,z\}_1,\overline y\rangle+\langle x,\overline{\{y,z,u\}_1}\rangle\\&=
      -\langle \langle x,\langle \overline u,\overline z\rangle\rangle,\overline y\rangle+\langle x,\langle \overline y,\langle \overline z,\overline u\rangle\rangle\rangle=\langle \langle x,\langle \overline z,\overline u\rangle\rangle,\overline y\rangle+\langle x,\langle \overline y,\langle \overline z,\overline u\rangle\rangle\rangle\\&=\langle \langle x,\overline y\rangle,\langle \overline z,\overline u\rangle\rangle.
    \end{align*}
The calculations above show that any $T=[x,y]\in D(V)$ gives rise to two elements in $\hat L$:  $T'=\langle x,\overline y\rangle$ or $\overline{T'}=\langle \overline x,\overline y\rangle$
 depending on its position on the left or on the right in the product. Extend this notation  to $D(V)$ by linearity.

 Let us prove that $\text{TKK}(V)$ is a Leibniz algebra:
      for $x,z,v\in V^\sigma$, $y,u,w\in V^{-\sigma}$ and $F,G,H\in D(V)$ we have
    \begin{align*}
      &[[x+G+y,z+H+u],v+F+w]=\langle\langle x+G'+y,\overline{z}+\overline{H'}+\overline{u}\rangle, \overline{v}+\overline{F'}+\overline{w}\rangle\\&=
      \langle\langle x+G'+y,\overline{v}+\overline{F'}+\overline{w}\rangle, \overline{z}+\overline{H'}+\overline{u}\rangle+\langle x+G'+y,\langle\overline{z}+\overline{H'}+\overline{u}, \overline{v}+\overline{F'}+\overline{w}\rangle\rangle\\&
      =[[x+G+y,v+F+w],z+H+u]+[x+G+y,[z+H+u,v+F+w]].
    \end{align*}
\end{proof}

This construction generalizes the  Tits-Kantor-Koecher construction for Jordan dialgebras  given by Gubarev and Kolesnikov in \cite{TKKDialgebras}. Since their construction was done for left structures, we first need  to relate left and right Leibniz algebras and left and right Jordan dialgebras. Indeed, if $L$ and $J$ are a left Leibniz algebra and a left Jordan dialgebra with product $[\ ,\ ]$ and $\bullet$ respectively, then the $\phi$-modules $L$ and $J$   with  products $[x,y]_{op}= [y,x]$ and $x \bullet_{op} y = y \bullet x$ are a right Leibniz algebra and a right Jordan dialgebra respectively, denoted by $L^{op}$ and $J^{op}$.

Let $J$ be a right Jordan dialgebra, let $J^{op}$ be the corresponding left Jordan dialgebra and  consider  the (left) Leibniz algebra $T(J^{op})=J^+\oplus S_0(J)\oplus J^-$, where $S_0(J)$ is spanned by operators $L_{x\bullet y}$ and $[L_x,L_y]$ (see \cite[\S 5.1]{TKKDialgebras}) and  with product
$[\ ,\ ]_t$ given in  \cite[\S 6.1]{TKKDialgebras}.
In the notation of \cite{TKKDialgebras}, $ [y,x]_t = L_{y \bullet_{op} x} + [L_y, L_x] = L_{x\bullet y} + [L_y,L_x]$, $[y,x]_t^* =-L_{x\bullet y} + [L_y,L_x]$ and
\begin{align*}
    L_{x\bullet y} \vdash z & = (x \bullet y) \bullet_{op} z = z \bullet (x \bullet y)
    \\
    L_{x\bullet y} \dashv z & = z \bullet_{op} (x \bullet y) = (x \bullet y) \bullet z
    \\
    [L_y,L_x] \vdash z & = y\bullet_{op} (x \bullet_{op} z) - x \bullet_{op} (y \bullet_{op} z) = (z \bullet x) \bullet y - (z \bullet y) \bullet x
    \\
    [L_y,L_x] \dashv z & = y \bullet_{op} (z \bullet_{op} x) - (z \bullet_{op} y) \bullet_{op} x = (x \bullet z) \bullet y - x \bullet ( y \bullet z)
\end{align*}
for any $x, z \in J^+$ and $y \in J^-$. Thus, if we rewrite  some of the products of \cite[Theorem 25]{TKKDialgebras} in terms of the Jordan product $\bullet$ we get:
\begin{align*}
    [z,[x,y]_{op}]_{op} &= [[y,x]_t,z]_t = [y,x]^*_t \vdash z = (-L_{x\bullet y} + [L_y,L_x]) \vdash z
    \\
    &=-z \bullet (x \bullet y) + (z \bullet x) \bullet y - (z \bullet y) \bullet x
    \\ & \\
    [u,[x,y]_{op}]_{op} &= [[y,x]_t,u]_t = [y,x]_t \vdash u = (L_{x\bullet y} + [L_y,L_x]) \vdash u
    \\
    &= u \bullet (y \bullet x) + (u \bullet x) \bullet y - (u \bullet y) \bullet x
    \\ & \\
    [[x,y]_{op},z]_{op} &= [z,[y,x]_t]_t = -[y,x]^*_t \dashv z =(L_{x\bullet y} - [L_y,L_x])\dashv z
    \\
    &=(x \bullet y) \bullet z - (x \bullet z) \bullet y + x \bullet ( y \bullet z)
    \\ & \\
    [[x,y]_{op},u]_{op} &= [u,[y,x]_t]_t = -[y,x]_t\dashv u = (-L_{x\bullet y} - [L_y, L_x]) \dashv u
    \\
    &=-(x \bullet y) \bullet u - (x \bullet u) \bullet y + x \bullet (y \bullet u)
\end{align*}
for any $x, z \in J^+$ and $y, u \in J^-$.

\begin{cor}
    Let $J$ be a right Jordan dialgebra and let us consider the (right) Jordan pair disystem $V=(J,J)$. Then
    $ {\rm TKK}(V)$ is isomorphic as a (right) Leibniz algebra to the (right) Leibniz subalgebra $L=J^+\oplus ([J^+,J^-]+[J^-,J^+])\oplus J^-$ of $T(J^{op})^{op}=J^+\oplus S_0(J)\oplus J^-$ generated by $J^+$ and $J^-$.
\end{cor}
\begin{proof} Consider the map
    \begin{align*}
       \Psi: {\rm TKK}(V)  & \longrightarrow J^+\oplus ([J^+,J^-]+[J^-,J^+])\oplus J^-
        \\
        x  & \mapsto x
        \\
       [ x, y] & \mapsto [x, y]_{op},
    \end{align*}
    for every $x \in V^\sigma$ and $y \in V^{-\sigma}$. By definition of the products in ${\rm TKK}(V)$ and in $L$, $\Psi$ is bijective.  Moreover, by Proposition \ref{doubleJ}
    \begin{align*}
        \{x,y,z\}_1&= x\bullet(y\bullet z)-(x\bullet z)\bullet y+(x\bullet y)\bullet z,
        \\
        \{x,y,z\}_2&= (y \bullet z) \bullet x + (y \bullet x) \bullet z - y \bullet (x \bullet z),
    \end{align*}
    and by the  above calculations  $\Psi$ is a homomorphism of (right) Leibniz algebras.
\end{proof}

\section{Homotopes of (right) Jordan pair disystems}

\begin{defn}
Let $V=(V^+,V^-)$ be a (right) Jordan pair disystem and let $a\in V^{-\sigma}$. In the $\phi$-module $V^\sigma$ we define the following product
$$
x\bullet y=\{x,a,y\}_1
$$
for every $x,y\in V^{\sigma}$. We denote $V^{(a)}$ the $\phi$-module $V^\sigma$ with this new product $\bullet$ and call it the {\bf homotope of $V$ at the element $a$}.
 \end{defn}

This product has already been introduced in the context of Jordan dialgebras by Alhefthi, Siddiqui and Jamjoom in 2023, see \cite[\S 3]{ASJ}. In that paper, given a Jordan dialgebra $J$  the authors considered the trilinear product  $\{\ ,\ , \ \}_1$ (see  formula (2) in \cite[pag. 3]{ASJ}) and for every $a\in J$ they defined a bilinear product $\bullet$ in $J$ by $x\bullet y=\{x,a,y\}_1$, calling the resulting $(J,\bullet)$ the $a$-homotope of $J$. In \cite[Proposition 6]{ASJ} they showed that the $a$-homotope of a Jordan dialgebra satisfied axioms (JD1) and (JD2) of the definition of Jordan dialgebra. Nevertheless, they were not able to show axiom (JD3) and left it as an open problem in \cite[Remark 2]{ASJ}.

In this section we will show that $V^{(a)}$ is a Jordan dialgebra, proving in particular that the homotope of $J$ defined by Alhefthi, Siddiqui and Jamjoom is in fact a Jordan dialgebra.

\begin{prop}
Given a  (right) Jordan pair disystem and an element $a\in V^{-\sigma}$, the homotope of $V$ at $a$ is a Jordan dialgebra.
\end{prop}

\begin{proof}
We will move to the (right) Leibniz algebra $L={\rm TKK}(V)=L_{\sigma 1}\oplus L_0\oplus L_{-\sigma 1}$. Recall that $L_{\sigma 1}=V^{\sigma}$ and $a\in L_{-\sigma 1}$.
Notice that for any $y\in L_{\sigma1}$ and any $b\in L_{-\sigma1}$, both $y$ and $b$ are $Q$-Jordan elements of $L$ by the finite $\mathbb{Z}$-grading of $L$,  and  therefore, by \ref{QJordan}, $$BYB^2=B^2YB \ \hbox{ and }\ Y^2BY=YBY^2.\eqno{(1)}$$

 Let us prove the right commutativity, i.e.,  $x\bullet(y\bullet z)=x\bullet(z\bullet y)$ for every $x,y\in V^\sigma$:
\begin{align*}
 &x\bullet(y\bullet z)=[x,[a,[y,[a,z]]]]\\&=(x)(-A^2  Z  Y + A  Y  A  Z - A  Y  Z  A +
 A  Z  A  Y + A  Z  Y  A - Y  A  Z  A +
 Y  Z  A^2 - Z  A  Y  A)\\&= (x)(-A^2  Z  Y + A  Y  A  Z  +
 A  Z  A  Y )
\end{align*}
since the terms $(x) A  Y  Z  A=0$,  $(x)A  Z  Y  A=0$, $(x) Y  A  Z  A=0 $, $(x) Y  Z  A^2=0 $ and $(x) Z  A  Y  A=0$ by the  finite $\mathbb{Z}$-grading of $L$.
\begin{align*}
 &x\bullet(z\bullet y)=[x,[a,[z,[a,y]]]]\\&=(x)(-A^2  Y  Z + A  Z  A  Y - A  ZY  A +
 A  Y  A  Z + A  YZ  A - Z  A  Y  A +
 ZY  A^2 - Y  A  Z  A)\\& =(x)(-A^2  YZ + A  Z  A  Y  +
 A  Y  A  Z )\qquad \hbox{(by the  finite $\mathbb{Z}$-grading of $L$)}
\end{align*}
Thus
\begin{align*}
&x\bullet(y\bullet z)-x\bullet(z\bullet y)\\&=(x)(-A^2  Z  Y + A  Y  A  Z  +
 A  Z  A  Y -(-A^2  YZ + A  Z  A  Y  +
 A  Y  A  Z ))\\&=(x)(-A^2  Z  Y +A^2  YZ)=0.
\end{align*}

 Let us prove the right Jordan identity, i.e,  $(x\bullet y)\bullet y^{2}=(x\bullet y^{2})\bullet y$ for every $x,y,z\in V^\sigma$:
\begin{align*}
&(x\bullet y)\bullet y^{2}=[[x,[a,y],[a,[y,[a,y]]]]]\\&=(x)(-A  Y  A^2  Y^2 + 2 A  Y  A  Y  A  Y -
 2 A  Y^2  A  Y  A + A  Y^3  A^2 \\&+
 Y  A^3  Y^2 - 2 Y  A^2  Y  A  Y +
 2 Y  A  Y  A  Y  A - Y  A  Y^2  A^2)=(x)(-A  Y  A^2  Y^2 + 2 A  Y  A  Y  A  Y  \\&- 2 Y  A^2  Y  A  Y +
 2 Y  A  Y  A  Y  A )\qquad \hbox{(by the  finite $\mathbb{Z}$-grading of $L$)}
\end{align*}
\begin{align*}
&(x\bullet y^{2})\bullet y=[[x,[a,[y,[a,y]]]],[a,y]]\\&=(x)(-A^2  Y^2  A  Y + A^2  Y^3  A + 2 A  Y  A  Y  A  Y - 2 A  Y  A  Y^2  A \\&-
 2 Y  A  Y  A^2  Y + 2 Y  A  Y  A  Y  A +
 Y^2  A^3 Y - Y^2  A^2  Y  A)\\&=(x)(-A^2  Y^2  A  Y  +
 2 A  Y  A  Y  A  Y  -
 2 Y  A  Y  A^2  Y + 2 Y  A  Y  A  Y  A) \\&\qquad \hbox{(by the  finite $\mathbb{Z}$-grading of $L$)}
\end{align*}

Since $Y^2AY=YAY^2$ and $A^2YA=AYA^2$ by (1) , we have
\begin{align*}
&(x\bullet y)\bullet y^{2}-(x\bullet y^{2})\bullet y\\&=(x)(-A  Y  A^2  Y^2 + 2 A  Y  A  Y  A  Y  - 2 Y  A^2  Y  A  Y +
 2 Y  A  Y  A  Y  A\\&-(-A^2  Y^2  A  Y
 +
 2 A  Y  A  Y  A  Y  -
 2 Y  A  Y  A^2  Y + 2 Y  A  Y  A  Y  A) )\\&=(x)(-A  Y  A^2  Y^2- 2 Y  A^2  Y  A  Y+A^2  Y^2  A  Y +2 Y  A  Y  A^2  Y)\\&
 =(x)(-A  Y  A^2  Y^2- 2 Y  A  Y  A^2  Y+A^2  Y  A  Y^2 +2 Y  A  Y  A^2  Y)\\&=(x)(-A  Y  A^2  Y^2+A  Y  A^2  Y^2 )=0.
\end{align*}

 Let us prove Osborn identity, i.e.,  for every $x,y,z\in V^\sigma$ $$(x\bullet y^{2})\bullet z-x\bullet(y^{2}\bullet z)=2\left(((x\bullet y)\bullet z)\bullet y-(x\bullet(y\bullet z))\bullet y\right):$$
\begin{align*}
& (x\bullet y^{2})\bullet z=[[x,[a,[y,[a,y]]]],[a,z]]=(x)[A,[Y,[A,Y]][A,Z]\\&
 =(x)(-A^2  Y^2  A  Z + A^2  Y^2  Z  A +
 2 A  Y  A  Y  A  Z - 2 A  Y  A  Y  Z  A -
  2 Y  A  Y  A  A  Z \\&+ 2 Y  A  Y  A  Z  A +
 Y  Y  A^3  Z - Y  Y  A^2  Z  A)\\&=(x)(-A^2  Y^2  A  Z + A^2  Y^2  Z  A +
 2 A  Y  A  Y  A  Z - 2 A  Y  A  Y  Z  A) \\&\qquad \hbox{(by the  finite $\mathbb{Z}$-grading of $L$).}
 \end{align*}
 \begin{align*}
& x\bullet(y^{2}\bullet z)=[x,[a,[[y,[a,y]],[a,z]]]]=(x)[A,[[Y,[A,Y]],[A,Z]]]\\&=(x)(-A^2  Y^2  A  Z + A^2  Y^2  Z  A +
 A^2  Z  A  Y^2 - 2 A^2  Z  Y  A  Y +
 A^2  Z  Y^2  A \\&+ 2 A  Y  A  Y  A  Z -
 2 A  Y  A  Y  Z  A - A  Y^2  A^2  Z +
 2 A  Y^2  A  Z  A - A  Y^2  Z  A^2 \\&-
 A  Z  A^2  Y^2 + 2 A  Z  A  Y  A  Y -
 2 A  Z  A  Y^2  A + 2 A  Z  Y  A  Y  A -
 A  Z  Y^2  A^2\\&
  - 2 Y  A  Y  A  Z  A +
 2 Y  A  Y  Z  A^2 + Y^2  A^2  Z  A -
 Y^2  A  Z  A^2 + Z  A^2  Y^2  A \\&-
 2 Z  A  Y  A  Y  A + Z  A  Y^2  A^2)\\&=(x)(-A^2  Y^2  A  Z + A^2  Y^2  Z  A +
 A^2  Z  A  Y^2 - 2A^2  Z  Y  A  Y  + 2A  Y  A  Y  A  Z\\& -
 2 A  Y  A  Y  Z  A  -
 A  Z  A^2  Y^2 + 2 A  Z  A  Y  A  Y)\\&\qquad \hbox{(by the  finite $\mathbb{Z}$-grading of $L$).}
\end{align*}

Since $A^2ZA=AZA^2$ by (1), we have
\begin{align*}
 &(x\bullet y^{2})\bullet z- x\bullet(y^{2}\bullet z)\\&=(x)(-A^2  Y^2  A  Z + A^2  Y^2  Z  A +
 2 A  Y  A  Y  A  Z - 2 A  Y  A  Y  Z  A\\&-(-A^2  Y^2  A  Z + A^2  Y^2  Z  A +
 A^2  Z  A  Y^2 - 2A^2  Z  Y  A  Y  \\&+ 2A  Y  A  Y  A  Z -
 2 A  Y  A  Y  Z  A  -
 A  Z  A^2  Y^2 + 2 A  Z  A  Y  A  Y ))\\&=(x)(-A^2ZAY^2+2A^2ZYAY+AZA^2Y^2-2AZAYAY)\\&=(x)(2A^2ZYAY-2AZAYAY).
 \end{align*}

On the other hand, and taking into account the finite  $\mathbb{Z}$-grading of $L$,
 \begin{align*}
&((x\bullet y)\bullet z)\bullet y=[[[x,[a,y]],[a,z]],[a,y]]=(x)[A,Y][A,Z][A,Y]\\&=(x)(A  Y  A  Z  A  Y - A  Y  A  Z  Y  A -
 A  Y  Z  A^2  Y + A  Y  Z  A  Y  A \\&-
 Y  A^2  Z  A  Y + Y  A^2  Z  Y  A +
 Y  A  Z  A^2  Y - Y  A  Z  A  Y  A)\\&=(x)(A  Y  A  Z  A  Y - A  Y  A  Z  Y  A -
 A  Y  Z  A^2  Y + A  Y  Z  A  Y  A).\\
&
 \\
&(x\bullet(y\bullet z))\bullet y=[[x,[a,[y,[a,z]]]],[a,y]]\\&=(x)(-A^2  Z  Y  A  Y + A^2  Z  Y^2  A +
 A  Y  A  Z  A  Y - A  Y  A  Z  Y  A \\&-
 A  Y  Z  A^2  Y + A  Y  Z  A  Y  A +
 A  Z  A  Y  A  Y - A  Z  A  Y^2  A \\&+
 A  Z  Y  A^2  Y - A  Z  Y  A  Y  A -
 Y  A  Z  A^2  Y +Y  A  Z  A  Y  A \\&+
 Y  Z  A^3  Y - Y  Z  A^2  Y  A -
 Z  A  Y  A^2  Y + Z  A  Y  A  Y  A)\\&=
 (x)(-A^2  Z  Y  A  Y  +
 A  Y  A  Z  A  Y - A  Y  A  Z  Y  A -
 A  Y  Z  A^2  Y \\&+ A  Y  Z  A  Y  A +
 A  Z  A  Y  A  Y).
\end{align*}

Thus
\begin{align*}
 &((x\bullet y)\bullet z)\bullet y-(x\bullet(y\bullet z))\bullet y=(x)(A  Y  A  Z  A  Y - A  Y  A  Z  Y  A\\& -
 A  Y  Z  A^2  Y + A  Y  Z  A  Y  A-(- A^2  Z  Y  A  Y  +
 A  Y  A  Z  A  Y \\&- A  Y  A  Z  Y  A -
 A  Y  Z  A^2  Y + A  Y  Z  A  Y  A +
 A  Z  A  Y  A  Y))\\&=(x)(A^2  Z  Y  A  Y-A  Z  A  Y  A  Y),
 \end{align*}
 which implies that
 $$(x\bullet y^{2})\bullet z- x\bullet(y^{2}\bullet z)=2((x\bullet y)\bullet z)\bullet y-(x\bullet(y\bullet z))\bullet y).$$

\end{proof}

\section{Inner ideals and subquotients in (right) Leibniz algebras}

\begin{defn}
    Let $L$ be a (right) Leibniz algebra and let $B\subseteq L$ be a submodule of $L$. We say that $B$ is an \textbf{inner ideal of $L$} if
    $$
        [[L,B],B]+[B,[B,L]]\subseteq B.
    $$
    Moreover, if $B$ satisfies $[B,B]=\{0\}$, we say that $B$ is a \textbf{abelian inner ideal of} $L$.

Notice that any element in an abelian inner ideal is a $Q$-Jordan element. Indeed, if $B$ is an abelian inner ideal of $L$ then, for any $b\in B$, we have that $\text{ad}_{b}^{3}(x)=[[[x,b],b],b]\in[B,B]=\{0\}$, for every $x\in L$, which implies that every element in $B$ is a $Q$-Jordan element of $L$.
\end{defn}

\begin{rem}
In \cite[Definition 48]{QJAlg} the authors defined  inner ideals only requiring  $[[L,B],B]\subseteq B$. Our definition is more restrictive and will allow us, as we will see in this section, to give a notion of subquotient in the context of Leibniz algebras.
\end{rem}

\begin{lem}
    Let $L$ be a (right) Leibniz algebra and let $a$ be a $Q$-Jordan element of $L$. Then $[[L,a],a]$ is an abelian inner ideal of $L$.
\end{lem}

\begin{proof}
    Let $a$ be a $Q$-Jordan element of $L$ and  let us consider $B=[[L,a],a]$. It is easy to prove that $B$ is a submodule of $L$ and $[B,B]=\{0\}$. We need to see that
    $$
        [[L,B],B]+[B,[B,L]]\subseteq B.
    $$
  By \cite[Lemma 49]{QJAlg}, we have $[[L,B],B]\subseteq B$. Let us show that $[B,[B,L]]\subseteq B$.  For every $x,y,z\in L$
\begin{align*}
 [[[x,a],a],&[[[y,a],a],z]]=(x)A^2[[[Y,A],A],Z]
 \\
 &=(x)A^2 ((YA^2+A^2Y-2AYA)Z-Z(YA^2+A^2Y-2AYA))
 \\
 &=(x)(-A^2ZYA^2+2AZAYA^2)\in B
\end{align*}
because $A^3=0$, $A^2YA^2=A^2ZA^2=0$ and $A^2ZAYA=AZA^2YA=AZAYA^2$ by \ref{QJordan}.
\end{proof}

\begin{defn}\label{kernel}
Let $L$ be a (right) Leibniz algebra  and let
$B$ be an inner ideal of $L$. We define the $\phi$-module
$$\text{Ker}_{L}B:=\{x\in L:[[x,B],B]=[B,[B,x]]=0\}$$
and we call it the \textbf{kernel of $B$ in $L$}.
\end{defn}

Let us see a technical lemma about kernels.

\begin{lem}\label{lem:ker}
    Let $L$ be a (right) Leibniz algebra and let $B$ an inner ideal of $L$. Then
\begin{align*}
&[{\rm Ker}_{L}B,[B,L]]\ +\ [[B,L], {\rm Ker}_{L}B]\ +\ [L,[B,{\rm Ker}_{L}B]]\subseteq  {\rm Ker}_{L}B \hbox{ and } \\&[L,B]\ +\ [B,L]\ +\ [[L,B],[L,B]] \subseteq {\rm Ker}_{L}B.
\end{align*}
 \end{lem}

\begin{proof} Firstly, let us note that $[B,[{\rm Ker}_{L}B,B]]=\{0\}$ and $[[B,{\rm Ker}_{L}B],B]=\{0\}$.

\noindent $\bullet$ $[{\rm Ker}_{L}B,[B,L]]\subseteq {\rm Ker}_{L}B$:  for every $k\in {\rm Ker}_{L}B$, $a,b,c\in B$ and $x\in L$,
\begin{align*}
  [[[k,[&b,x]],a],c]=[[[k,a],[b,x]],c]+[[k,[[b,x],a]],c]\\&=[[[k,a],c],[b,x]]+[[k,a],[[b,x],c]]\\&=-[[k,a],[c,[b,x]]]=0,\\
  [c,[a&,[k,[b,x]]]]=[c,[[a,k],[b,x]]]+[c,[k,[a,[b,x]]]]\\&=[[c,[a,k]],[b,x]]-[[c,[b,x]],[a,k]]\\&-[c,[[a,[b,x]],k]]=0.
\end{align*}

\noindent $\bullet$ $[[B,L], {\rm Ker}_{L}B]\subseteq {\rm Ker}_{L}B$: for every $k\in {\rm Ker}_{L}B$, $a,b,c\in B$ and $x\in L$,
\begin{align*}
  [[[[b,&x],k],a],c]=[[[b,x],a],k],c]=0,\\
  [c,[&a,[[b,x],k]]]=[c,[[a,[b,x]],k]]+[c,[[b,x],[a,k]]]\\&=0+[[c,[b,x]],[a,k]]-[[c,[a,k]],[b,x]]=0.
\end{align*}

\noindent $\bullet$ $[L,[B,{\rm Ker}_{L}B]]\subseteq {\rm Ker}_{L}B$: for every $k\in {\rm Ker}_{L}B$, $a,b,c\in B$ and $x\in L$,
\begin{align*}
  [[[x,[&b,k]],a],c]=[[[x,a],[b,k]],c]+[[x,[[b,k],a]],c]\\&=[[[x,a],c],[b,k]]+[[x,a],[[b,k],c]]\\&=-[[x,a],[c,[b,k]]]=0,\\
  [c,[a,[x,[b,k]]]]&=[c,[[a,x],[b,k]]]+[c,[x,[a,[b,k]]]]\\&=[[c,[a,x]],[b,k]]-[[c,[b,k]],[a,x]]+0=0.
 \end{align*}

 \noindent $\bullet$ $[L,B]\subseteq {\rm Ker}_{L}B$: for every  $a,b,c\in B$ and $x\in L$,
$$
  [[[x,b],a],c]=0 \qquad\hbox{and}\qquad
  [c,[a,[x,b]]]=-[c,[[x,b],a]]=0
$$

 \noindent $\bullet$ $[B,L]\subseteq {\rm Ker}_{L}B$: for every  $a,b,c,d\in B$ and $x,y\in L$,
$$
  [c,[a,[b,x]]]=0, \qquad\hbox{and}\qquad
    [[[b,x],a],c]=[[b,[x,a]],c]=[b,[[x,a],c]]=0.
$$

 \noindent $\bullet$ $[[L,B],[L,B]]\subseteq {\rm Ker}_{L}B$: for every  $a,b,c,d\in B$ and $x,y\in L$,
\begin{align*}
  [c,[a,[[b,x],[d,y]]]]&=[c,[[a,[b,x]],[d,y]]]]+[c,[[b,x],[a,[d,y]]]]=0, \\
    [[[[b,x],[d,y]],a],c]&=[[[[b,x],a],[d,y]],c]+[[[b,x],[[d,y],a]],c]=0.
\end{align*}
\end{proof}

\begin{thm}
Let $L$ be a (right) Leibniz algebra and let $B$ be an abelian inner ideal of $L$.
Then the pair of $\phi$-modules \textup{$V=(B,L/\text{Ker}_{L}B)$}
    has structure of (right) Jordan pair disystem with products
    \[
    \{b_{1},\overline{x},b_{2}\}_1= [b_{1},[x,b_2]]\text{, \qquad}\{\overline{x},b,\overline{y}\}_1= \overline{[x,[b,y]]},
    \]
    \[
    \{b_{1},\overline{x},b_{2}\}_2=-[[x,b_{1}],b_{2}]\text{, \qquad}\{\overline{x},b,\overline{y}\}_2=-\overline{[[b,x],y]},
    \]
    for every $b_{1},b_{2}\in B$ and for every \textup{$\overline{x},\overline{y}\in L/\text{Ker}_{L}B$},
    where \textup{$\overline{x}=x+\text{Ker}_{L}B$} and \textup{$\overline{y}=y+\text{Ker}_{L}B$}. This (right) Jordan pair disystem will be called the {\rm \bf subquotient of $L$ induced by $B$}.
\end{thm}

\begin{proof} By Lemma \ref{lem:ker} we have that for  $b_1, b_2 \in B$, $x \in L$ and $k \in \text{Ker}_LB$.

\begin{itemize}
  \item $\{b_1, k,b_2\}_1=[b_1,[k,b_2]]=0$ with implies that the product $\{b_{1},\overline{x},b_{2}\}_1= [b_{1},[x,b_2]]$ is well defined.
  \item $\{b_1, k,b_2\}_2=-[[k,b_1],b_2]=0$ with implies that the product $\{b_{1},\overline{x},b_{2}\}_2= [[x,b_1],b_2]$ is well defined.
  \item $\{k, b_1,x\}_1=[k,[b_1,x]]\in \text{Ker}_{L}B$ and $\{x, b_1,k\}_1=[x,[k,b_1]]=-[x,[b_1,k]]\in\text{Ker}_{L}B$ with implies that the product $\{\overline{x},b_{1},\overline{y}\}_1= \overline{[x,[b_1,y]]}$ is well defined.
  \item $\{k, b_1,x\}_2=-[[b_1,k],x]\in \text{Ker}_{L}B$ and $\{x, b_1,k\}_2=-[[b_1,x],k]\in \text{Ker}_{L}B$ with implies that the product $\{\overline{x},b_{1},\overline{y}\}_2= -\overline{[[x,b_1],y]}$ is well defined.
\end{itemize}

Let us see that $V$ is a (right) Jordan pair disystem:
    Let us consider $b_{1},b_{2},b_{3}\in B$ and $\overline{x},\overline{y},\overline{z}\in L/\text{Ker}_{L}B$,
    where $\overline{x}=x+\text{Ker}_{L}B$, $\overline{y}=y+\text{Ker}_{L}B$
    and $\overline{z}=z+\text{Ker}_{L}B$.
        \begin{itemize}
        \item [(JPD1)]
        \begin{align*}
            \{b_1, \overline{x}, b_2\}_2 & = -[[x, b_1],b_2] =-[[x,b_2],b_1] - [x,[b_1,b_2]] = \{b_2, \overline{x}, b_1\}_2
        \end{align*}
        because $[b_1,b_2] = 0$. In an analogous way,
        \begin{align*}
            \{\overline{x}, b,\overline{y}\}_2 & = -\overline{[[b, x],y]} = -\overline{[[b,y],x]} - \overline{[b,[x,y]]} = \{\overline{y}, b,\overline{x}\}_2
        \end{align*}
        because $[b,[x,y]] \in [B,L] \subseteq \text{Ker}_LB$ by Lemma \ref{lem:ker}.

        \item [(JPD2)]
        \begin{align*}
        \{b_{1},\,&\overline{x},\{b_{2},\overline{y},b_{3}\}_1\}_1 =[b_1, [{x}, [b_2,[y,b_3]]]]= [b_1, [{x}, [[b_2,y],b_3]]]+ [b_1, [{x}, [y,[b_2,b_3]]]] \\&= [b_1,[x,[b_3,[y,b_2]]]]+0 = \{b_1, \overline{x}, \{ b_3, \overline{y}, b_2\}_1\}_1
        \end{align*}
        because $[b_{2},b_{3}]=0$. Analogously,
        \begin{align*}
            \{\overline{x},\,& b_1, \{ \overline{y}, b_2, \overline{z}\}_1\}_1  = \overline{[x, [b_1, [y, [b_2, z]]]]} = \overline{[x, [b_1, [[y, b_2], z]]]} +\overline{[x, [b_1, [b_2, [y, z]]]]}\\&=\overline{[x, [b_1, [z,[b_2,y]]]]}+\overline{0}=\{\overline{x}, b_1, \{\overline{z}, b_2, \overline{y}\}_1\}_1
        \end{align*}
        because $[x, [b_1, [b_2, [y, z]]]]\in[L,B] \subseteq \text{Ker}_LB$ by Lemma \ref{lem:ker}.

        \item [(JPD3)]
        \begin{align*}
            \{b_{1},&\{\overline{x},b_{2},\overline{y}\}_1,b_{3}\}_1 = [b_1,[[x,[b_2,y]],b_3]]=[b_1,[[[x,b_2],y],b_3]]+[b_1,[[b_2,[x,y]],b_3]]\\&=
            [b_1,[[y,[b_2,x]],b_3]]+0=\{b_1, \{\overline{y},b_2,\overline{x}\}_1, b_3\}_1
        \end{align*}
        because $[b_1,[[b_2,[x,y]],b_3]]  \in [B,B] = \{0\}$, and
        \begin{align*}
            \{\overline{x},&\{b_{1},\overline{y},b_{2}\}_1,\overline{z}\}_1 = \overline{[x, [[b_{1},[y,b_{2}]],z] ]} = \overline{[x, [[[b_{1},y],b_{2}]],z] ]}+
            \overline{[x, [[y,[b_{1},b_{2}]],z] ]}\\&=\overline{[x, [[b_{2},[y,b_{1}]],z] ]}=\{\overline{x}, \{b_2,\overline{y},b_1\}_1, \overline{z}\}_1
        \end{align*}
        because $[b_1,b_2] \in [B,B] = \{0\}$.

        \item [(JPD4)]
        \begin{align*}
            \{\{b_1,\,& \overline{x}, b_2\}_1, \overline{y}, b_3\}_2  = - [[y,[b_1,[x,b_2]]],b_3]= - [[y,[[b_1,x],b_2]],b_3]-[[y,[x,[b_1,b_2]]],b_3]\\
            &=   [[y,[[x,b_1],b_2]],b_3]+0 =\{\{b_1, \overline{x}, b_2\}_2, \overline{y}, b_3\}_2,
        \end{align*}
        because $[b_1,b_2] = 0$. On the other hand,
        \begin{align*}
            \{\{\overline{x},\,&b_1,\overline{y}\}_1,b_2, \overline{z}\}_2  = - \overline{[[b_2, [x,[b_1,y]]],z]}=- \overline{[[b_2, [[x,b_1],y]],z]}- \overline{[[b_2, [b_1,[x,y]]],z]} \\&= \overline{[[b_2, [[b_1,x],y]],z]}+\overline 0=\{\{\overline{x},b_1,\overline{y}\}_2,b_2, \overline{z}\}_2,
        \end{align*}
        because $- [[b_2,[b_1,[x,y]]],z] \in [B,L] \subseteq \text{Ker}_LB$ by Lemma \ref{lem:ker}.

        \item [(JPD5)]
        \begin{align*}
            \{\{b_1,\,& \overline{x}, b_2\}_1, \overline{y}, b_3\}_1 = [[b_1,[x,b_2]], [y, b_3]] \\&=[[b_1,[y,b_3]], [x, b_2]]+[b_1,[[x,[y,b_3]],b_2]]+[b_1,[x,[b_2,[y, b_3]]]]\\&=\{\{b_1,\overline{y},b_3\}_1,\overline{x},b_2\}_1- \{b_1, \{\overline{x}, b_3, \overline{y}\}_1,b_2\}_1+\{b_1, \overline{x}, \{b_2, \overline{y}, b_3\}_1\}_1
        \end{align*}
        On the other hand,
        \begin{align*}
            \{\{\overline{x},b_1,&\overline{y}\}_1, b_2, \overline{z}\}_1 =\overline{[[x,[b_1,y]], [b_2, z]]}\\&=\overline{[[x,[b_2,z]], [b_1, y]]}+\overline{[x,[[b_1, [b_2, z]],y]]}+\overline{[x,[b_1,[y, [b_2, z]]]]}
            \\&= \{\{\overline{x},b_2,\overline{z}\}_1, b_1, \overline{y}\}_1 - \{\overline{x},\{b_1,\overline{z},b_2\}_1, \overline{y} \}_1+ \{\overline{x}, b_1, \{\overline{y}, b_2, \overline{z}\}_1\}_1
        \end{align*}

        \item [(JPD6)]
        \begin{align*}
            \{\{b_1,\,& \overline{x}, b_2\}_2, \overline{y}, b_3\}_1  =-[[[x,b_1],b_2],[y,b_3]]\\&= -[[[x,[y,b_3]],b_1],b_2]-[[x,[b_1,[y,b_3]]],b_2]-[[x,b_1],[b_2,[y,b_3]]] \\&= [[[x,[b_3,y]],b_1],b_2]-[[x,[b_1,[y,b_3]]],b_2]-[[x,b_1],[b_2,[y,b_3]]]\\
            &= -\{b_1, \{\overline{x},b_3, \overline{y}\}_1, b_2\}_2+\{\{b_1, \overline{y}, b_3\}_1, \overline{x}, b_2\}_2 + \{b_1, \overline{x}, \{b_2, \overline{y}, b_3\}_1\}_2 ,
        \end{align*}
        and, similarly,
        \begin{align*}
            \{\{\overline{x},\,&b_1,\overline{y}\}_2,b_2, \overline{z}\}_1  = -\overline{[[[b_1,x],y],[b_2,z]]}\\&=
            -\overline{[[[b_1,[b_2,z]],x],y]}-\overline{[[b_1,[x,[b_2,z]]],y]}-\overline{[[[b_1,x],[y,[b_2,z]]]}\\
            &=\overline{[[[b_1,[z,b_2]],x],y]}-\overline{[[b_1,[x,[b_2,z]]],y]}-\overline{[[[b_1,x],[y,[b_2,z]]]}\\
            &= -\{\overline{x}, \{b_1,\overline{z},b_2\}_1, \overline{y}\}_2+\{\{\overline{x},b_2, \overline{z}\}_1, b_1, \overline{y}\}_2  + \{ \overline{x}, b_1, \{\overline{y}, b_2, \overline{z}\}_1\}_2.
        \end{align*}

        \item [(JPD7)]
        \begin{align*}
            \{b_1, \overline{x},& \{b_2, \overline{y}, b_3\}_1\}_1  = [b_1, [x, [b_2, [y, b_3]]]] \\& = [[b_1, x], [b_2, [y, b_3]]] =([b_1,x])(B_2YB_3-B_2B_3Y-YB_3B_2+B_3YB_2)
            \\&=[[[[b_1,x],b_2],y],b_3]-0-[[[[b_1,x],y],b_3],b_2]+[[[[b_1,x],b_3],y],b_2]
            \\&=[[[[b_1,x],b_2],[y,b_3]]-[[[[b_1,x],y],b_2],b_3]+[[[[b_1,x],b_3],[y,b_2]]
            \\&=[[b_1,[x,b_2]],[y,b_3]]-[[[[b_1,x],y],b_2],b_3]+[[[b_1,x],b_3],[y,b_2]]
            \\ &=\{\{b_1, \overline{x},b_2\}_1,\overline{y},b_3\}_1  - \{b_2, \{\overline{x}, b_1, \overline{y}\}_2, b_3\}_2 + \{\{b_1,\overline{x},b_3\}_1, \overline{y}, b_2\}_1,
        \end{align*}
        because $[[[b_1,x],b_2],b_3] \in [B,B] = \{0\}$ and $[b_i,b_j] \in [B,B]=\{0\}$ for all $i,j\in \{1,2,3\}$. Similarly,
        \begin{align*}
            \{\overline{x},\,& b_1, \{\overline{y}, b_2, \overline{z}\}_1\}_1  = \overline{[x,[b_1,[y,[b_2,z]]]]}\\&=^{(1)} \overline{[[x,b_1],[y,[b_2,z]]]}=\overline{([x,b_1])(YB_2Z-YZB_2-B_2ZY+ZB_2Y)}
            \\&=^{(2)}\overline{[[[[x,b_1],y],b_2],z]}-\overline 0-\overline{[[[[x,b_1],b_2],z],y]}+\overline{[[[[x,b_1],z],b_2],y]}
            \\&=^{(3)}\overline{[[[x,b_1],y],[b_2,z]]}-\overline 0-\overline{[[[[x,b_1],b_2],z],y]}+\overline{[[[x,b_1],z],[b_2,y]]}
            \\&=^{(4)}\overline{[[x,[b_1,y]],[b_2,z]]}-\overline 0-\overline{[[[[x,b_1],b_2],y],z]}+\overline{[[x,[b_1,z]],[b_2,y]]}\\
            &=\{\{\overline{x}, b_1,\overline{y}\}_1,b_2,\overline{z}\}_1  - \{\overline{y}, \{b_1, \overline{x}, b_2\}_2, \overline{z}\}_2 + \{\{\overline{x},b_1,\overline{z}\}_1, b_2, \overline{y}\}_1.
        \end{align*}
        $^{(1)}$ because $[[x,[y,[b_2,z]],b_1]]\in [L,B]\subseteq \text{Ker}_LB$, \\$^{(2)}$ because $[[[[x,b_1],y],z],b_2]\in [L,B]\subseteq \text{Ker}_LB$, \\$^{(3)}$ because $[[[[x,b_1],y],z],b_2],[[[[x,b_1],z],y],b_2]\in [L,B]\subseteq \text{Ker}_LB$, \\$^{(4)}$ because $[[[x,y],b_1],[b_2,z]], [[[x,z],b_1],[b_2,y]]\in [[L,B],[L,B]]\subseteq \text{Ker}_LB$ and $[[[x,b_1],b_2],[y,z]]\in [B,L]\subseteq \text{Ker}_LB$.

        \item [(JPD8)]
        \begin{align*}
            \{b_1,\overline{x},& \{b_2, \overline{y}, b_3\}_1\}_2  =  [[x,b_1],[b_2,[b_3,y]]] \\&=([x,b_1](B_2B_3Y-B_2YB_3-B_3YB_2+YB_3B_2)
            \\&=0-[[[[x,b_1],b_2],y],b_3]-[[[[x,b_1],b_3],y],b_2]+[[[[x,b_1],y],b_3],b_2]
            \\&=-[[[x,b_1],b_2],[y,b_3]]+[[[x,[b_1,y]],b_3],b_2]-[[[x,b_1],b_3],[y,b_2]]
            \\&=-[[[x,b_1],b_2],[y,b_3]]+[[[x,[b_1,y]],b_2],b_3]-[[[x,b_1],b_3],[y,b_2]]
            \\&= \{\{b_1,\overline{x},b_2\}_2, \overline{y}, b_3\}_1 -\{b_2,\{\overline{x},b_1,\overline{y}\}_1, b_3\}_2 + \{\{b_1,\overline{x},b_3\}_2,\overline{y},b_2\}_1.
        \end{align*}
        because $[[[x,b_1],b_2],b_3] \in [B,B] = \{0\}$ and $[b_i,b_j] \in [B,B]=\{0\}$ for all $i,j\in \{1,2,3\}$. Similarly,
        \begin{align*}
            \{\overline{x},b_1,& \{\overline{y}, b_2, \overline{z}\}_1\}_2  = \overline{[[b_1,x],[y,[z,b_2]]]}\\&= \overline{([b_1,x])(YZB_2-YB_2Z-ZB_2Y+B_2ZY)}
            \\&=^{(1)} \overline{0}-\overline{[[[[b_1,x],y],b_2],z]}-\overline{[[[[b_1,x],z],b_2],y]}+\overline{[[[[b_1,x],b_2],z],y]}
            \\&=^{(2)} \overline{0}-\overline{[[[b_1,x],y],[b_2,z]]}+\overline{[[[b_1,[x,b_2]],z],y]}-\overline{[[[[b_1,z],x],b_2],y]}
            \\&=^{(3)} \overline{0}-\overline{[[[b_1,x],y],[b_2,z]]}+\overline{[[[b_1,[x,b_2]],y],z]}-\overline{[[[b_1,z],x],[b_2,y]]}
            \\&= \{\{\overline{x},b_1,\overline{y}\}_2, b_2, \overline{z}\}_1 -\{\overline{y},\{b_1,\overline{x},b_2\}_1, \overline{z}\}_2 + \{\{\overline{x},b_1,\overline{z}\}_2,b_2,\overline{y}\}_1.
        \end{align*}
         $^{(1)}$ because $[[[[b_1,x],y],z],b_2]\in [L,B]\subseteq \text{Ker}_LB$, \\$^{(2)}$ because $[[[[b_1,x],y],z],b_2]\in [L,B]\subseteq \text{Ker}_LB$, $[b_1,b_2]=0$ \\and $[[[b_1,[z,x]],b_2],y]\in [B,L]\subseteq \text{Ker}_LB$, \\$^{(3)}$ because $[[b_1,[x,b_2]],[y,z]]\in [B,L]\subseteq \text{Ker}_LB$ \\and $[[[[b_1,z],x],y],b_2]\in [L,B]\subseteq \text{Ker}_LB$.
    \end{itemize}
\end{proof}

Recall that  R. Felipe and R. Velasquez attached a Jordan dialgebra  $L_a$ to a (right) Leibniz algebra $L$ via a $Q$-Jordan element $a$, see Remark \ref{localRaules}.
Their definition of the kernel of a $Q$-Jordan element $a$ is
$${\rm ker}_L\, \{a\}=\{z\in L\ :\  [[z,a],a]=0 \}$$
and the product in the Jordan dialgebra $L_a$ is given by $\bar x\circ \bar y=\frac 12\overline{[x,[y,a]]}$ for every $\bar x,\bar y\in L_a=L/{\rm ker}_L\, \{a\}$.
Let us see that our construction of subquotient extends their construction when dealing with the abelian inner ideal $[[L,a],a]$ and the kernel of the $Q$-Jordan element $a$ coincides with the kernel of the abelian inner ideal $[[L,a],a]$ in the sense of \ref{kernel}. In general, we always have that ${\rm ker}_L\, \{a\}\subseteq {\rm Ker}_L\,[[L,a],a]$, as we will see in the next lemma:

\begin{lem}\label{kernelscontained}
  If $L$ is a (right) Leibniz algebra, $a\in L$ is a  $Q$-Jordan element and we consider the abelian inner ideal $B=[[L,a],a]$, then ${\rm ker}_L\, \{a\}\subseteq {\rm Ker}_L\,B$.
\end{lem}
\begin{proof}
Let $a\in L$ be a  $Q$-Jordan element. Recall that
\begin{align*}
{\rm ker}_L\, \{a\}=\{z\in L\ :&\  [[z,a],a]=0 \}\\
{\rm Ker}_L\,B=\{z\in L\ :\ &[[z,[[x,a],a]],[[y,a],a]]=0 \hbox{ and }\\& [[[[x,a],a]],[[[y,a],a],z]]=0 \hbox{ for every $x,y\in L$}\}.
\end{align*}
In terms of operators,
\begin{align*}
[[z,[[&x,a],a]],[y,a],a]=(z)(A^2X+XA^2-2AXA)(A^2Y+YA^2-2AYA)\\&=(z)(A^2XYA^2)
\end{align*}
because $A^3=0$, $A^2XA^2=0$, $A^2YA^2=0$ and $A^2XAYA=AXA^2YA=AXAYA^2$ by \ref{QJordan}(1). Thus, if $z\in {\rm ker}_L\, \{a\}$, $(z)A^2=0$, so $[[z,[[x,a],a]],[[y,a],a]]=0$.

Moreover,  $[[z,a],a]=0$ implies that $$0={\rm ad}_{[[z,a],a]}(v)=(v)(ZA^2+A^2Z-2AZA) \hbox{ for every $v\in L$.}$$
This means that $2AZA=ZA^2+A^2Z$; multiplying by $A$ on the right we obtain $2AZA^2=ZA^3+A^2ZA=AZA^2$, so $AZA^2=0$ and $A^2ZA=0$ again by \ref{QJordan}(1). In particular,
\begin{align*}
[[[[x,&a],a]],[[[y,a],a],z]]= (x)A^2((YA^2+A^2Y-2AYA)Z-Z(YA^2+A^2Y-2AYA))\\&=(x)(-A^2ZYA^2+2A^2ZAYA)\\&=(x)(-A^2ZYA^2)=(x)(ZA^2YA^2-2AZAYA^2)\\&=(x)(-2AZA^2YA)=0.
\end{align*}
\end{proof}

\begin{defn}
Given two (right) Jordan pair disystems $(V^+,V^-)$ and $(W^+,W^-)$, we say that the pair of $\phi$-linear maps  $$(\Phi_1, \Phi_2):(V^+,V^-)\to (W^+,W^-)$$ is a {\bf homomorphism of (right) Jordan pair disystems} if
\begin{align*}
&\ \Phi_1(\{x,y,z\}_1)=\{\Phi_1(x),\Phi_2(y),\Phi_1(z)\}_1,\\
&\ \Phi_1(\{x,y,z\}_2)=\{\Phi_1(x),\Phi_2(y),\Phi_1(z)\}_2,\\
&\ \Phi_2(\{x,y,z\}_1)=\{\Phi_2(x),\Phi_1(y),\Phi_2(z)\}_1, \\
&\ \Phi_2(\{x,y,z\}_2)=\{\Phi_2(x),\Phi_1(y),\Phi_2(z)\}_2
\end{align*}
for every $x,z\in V^{\sigma}$, $y\in V^{-\sigma}$, $\sigma=\pm$.
\end{defn}

\begin{thm}
    Let $L$ be a (right) Leibniz algebra, let $a\in L$ be a $Q$-Jordan element and consider the abelian inner ideal $B = [[L,a],a]$. Then there exists an epimorphism of (right) Jordan pair disystems between the double pair disystem $(L_a,L_a)$ and the subquotient $ (B, L/\textup{\text{Ker}}_LB)$.\\
    In particular, when $\textup{\text{ker}}_L\, \{a\}=\textup{\text{Ker}}_LB$, both (right) Jordan pair disystems are isomorphic.
\end{thm}

\begin{proof}
Let us define  $(\Phi_1, \Phi_2) \colon (L_a, L_a) \rightarrow (B, L/\textup{\text{Ker}}_LB)$  by $$\Phi_1( \bar{x} ) := \frac 14 [[x,a],a] \hbox{ and } \Phi_2(\bar{x}) := - x + \text{Ker}_LB$$
for $\bar x=x+{\rm ker}_L\, \{a\}\in L_a$. The maps $\Phi_1$ and $\Phi_2$
 are clearly onto and well defined by Lemma \ref{kernelscontained}.

To see that $(\Phi_1, \Phi_2)$ is a homomorphism of (right) Jordan pair disystems, recall that in the (right) Jordan pair disystem $(L_a,L_a)$ we have two trilinear products
\begin{align*}
 \{\bar{x},  \, \bar{y}, \, \bar{z}\}_1 &= (\bar{x}\circ(\bar{y}\circ \bar{z}) - (\bar{x}\circ \bar{z})\circ \bar{y} + (\bar{x}\circ \bar{y})\circ \bar{z})\\&
 =\frac 14 ( {[x,[[y,[z,a]],a]]} - {[[x,[z,a]],[y,a]]} +{[[x,[y,a]],[z,a]]})+{\rm ker}_L\, \{a\}\\
 \{\bar{x},  \, \bar{y}, \, \bar{z}\}_2&=((\bar{y}\circ\bar{z})\circ \bar{x} + (\bar{y}\circ \bar{x})\circ \bar{z} - \bar{y}\circ (\bar{x}\circ \bar{z}))\\&
 =\frac 14({[[y,[z,a]],[x,a]]} +{[[y,[x,a]],[z,a]]} - {[y,[[x,[z,a]],a]]})+{\rm ker}_L\, \{a\}
\end{align*}
for every $\bar x=x+{\rm ker}_L\, \{a\}$, $\bar y=y+{\rm ker}_L\, \{a\}$, $\bar z=z+{\rm ker}_L\, \{a\}\in L_a$. We will denote adjoint maps by capital letters and will make use of \ref{QJordan}(1) without explicitly mentioning it.

\noindent -- Let us check that $\Phi_1( \{\bar{x},  \, \bar{y}, \, \bar{z}\}_1)= \{\Phi_1(\bar{x}),  \, \Phi_2(\bar{y}), \, \Phi_1(\bar{z})\}_1$:
\begin{align*}
\bullet\ &\Phi_1( \{\bar{x},  \, \bar{y}, \, \bar{z}\}_1)=\frac 1{16}( [[[x,[[y,[z,a]],a]],a],a] - [[[[x,[z,a]],[y,a]],a],a] \\&+[[[[x,[y,a]],[z,a]],a],a]).\\
\bullet\ &\{\Phi_1(\bar{x}),  \, \Phi_2(\bar{y}), \, \Phi_1(\bar{z})\}_1=\{ \frac 14[[x,a],a],\, -y+\text{Ker}_LB,\, \frac 14 [[z,a],a]\}_1
\\&=-\frac 1{16}[ [[x,a],a], [y, [[z,a],a]]]
\end{align*}
In terms of operators acting on $x$,
\begin{align*}
   \bullet\  &[[[x,[[y,[z,a]],a]],a],a] - [[[[x,[z,a]],[y,a]],a],a] +[[[[x,[y,a]],[z,a]],a],a]\\&=(x)([Y,[Z,A],A]A^2-[Z,A][Y,A]A^2+[Y,A][Z,A]A^2)  \\
   &=(x)(AYAZA^2+AZAYA^2-A^2YZA^2-AZAYA^2+AYAZA^2)\\&=(x)(2AYAZA^2-A^2YZA^2),\\
  \bullet\  & -[ [[x,a],a], [y, [[z,a],a]]]=-(x)(A^2[Y,[[Z,A],A]])\\&=-(x)(A^2YZA^2-2A^2YAZA),
\end{align*}
proving that
$$
\Phi_1( \{\bar{x},  \, \bar{y}, \, \bar{z}\}_1)= \{\Phi_1(\bar{x}),  \, \Phi_2(\bar{y}), \, \Phi_1(\bar{z})\}_1.
$$

\noindent -- Let us check that $\Phi_1( \{\bar{x},  \, \bar{y}, \, \bar{z}\}_2)= \{\Phi_1(\bar{x}),  \, \Phi_2(\bar{y}), \, \Phi_1(\bar{z})\}_2$:

\begin{align*}
\bullet\ &\Phi_1( \{\bar{x},  \, \bar{y}, \, \bar{z}\}_2)=\frac 1{16}({[[[[y,[z,a]],[x,a]],a],a]} +{[[[[y,[x,a]],[z,a]],a],a]} \\&- {[[[y,[[x,[z,a]],a]],a],a]}).\\
\bullet\ &\{\Phi_1(\bar{x}),  \, \Phi_2(\bar{y}), \, \Phi_1(\bar{z})\}_2=\{ \frac 14[[x,a],a],\, -y+\text{Ker}_LB,\, \frac 14 [[z,a],a]\}_2
\\&=\frac 1{16}[[y,[[x,a],a]],[[z,a],a]].
\end{align*}
In terms of operators acting on $y$,
\begin{align*}
   \bullet\  &[[[[y,[z,a]],[x,a]],a],a] +[[[[y,[x,a]],[z,a]],a],a] - [[[y,[[x,[z,a]],a]],a],a]\\&
   =(y)([Z,A][X,A]A^2+[X,A][Z,A]A^2-[[X,[Z,A]],A]A^2)\\&
   =(y)(AZAXA^2+AXAZA^2-AXAZA^2-AZAXA^2+A^2ZXA^2)\\&
   =(y)A^2ZXA^2,\\
      \bullet\  &[[y,[[x,a],a]],[[z,a],a]]=(y)(XA^2+A^2X-2AXA)(ZA^2+A^2Z-2AZA)\\&
   =(y)A^2XZA^2=(y)A^2ZXA^2
\end{align*}
proving that $$\Phi_1( \{\bar{x},  \, \bar{y}, \, \bar{z}\}_2)= \{\Phi_1(\bar{x}),  \, \Phi_2(\bar{y}), \, \Phi_1(\bar{z})\}_2.$$

\noindent -- Let us check that $\Phi_2( \{\bar{x},  \, \bar{y}, \, \bar{z}\}_1)= \{\Phi_2(\bar{x}),  \, \Phi_1(\bar{y}), \, \Phi_2(\bar{z})\}_1$:
\begin{align*}
\bullet\ &\Phi_2( \{\bar{x},  \, \bar{y}, \, \bar{z}\}_1)=-\frac 14 ( {[x,[[y,[z,a]],a]]} - {[[x,[z,a]],[y,a]]} +{[[x,[y,a]],[z,a]]})+\text{Ker}_LB\\
=&-\frac 14 ({[x,[[y,[z,a]],a]]}+[x, [[y,a],[z,a]]]  )+ \text{Ker}_LB.\\
\bullet\ &\{\Phi_2(\bar{x}),  \, \Phi_1(\bar{y}), \, \Phi_2(\bar{z})\}_1=\{-x+ \text{Ker}_LB, \, \frac14[[y,a],a] , \, -z+\text{Ker}_LB\}_1\\&
=\frac14[x, [[[y,a],a], z]]+\text{Ker}_LB.
\end{align*}
Let us see that
$$
[x, [[[y,a],a], z]]+ {[x,[[y,[z,a]],a]]}+[x, [[y,a],[z,a]]]\in \text{Ker}_LB.
$$
Indeed,
\begin{align*}
&[x, [[[y,a],a], z]]+ {[x,[[y,[z,a]],a]]}+[x, [[y,a],[z,a]]]\\&=[x, [[[y,a],z]],a]]+[x,[[y,a],[a,z]]]+{[x,[[y,[z,a]],a]]}+[x, [[y,a],[z,a]]]\\
&=[x, [[[y,a],z]],a]]+{[x,[[y,[z,a]],a]]}\\&=[x,[[[[y,z],a],a]]]+[x,[[y,[a,z]],a]]+{[x,[[y,[z,a]],a]]}\\
&=[x,[[[[y,z],a],a]]]\in [L,B]\subseteq \text{Ker}_LB \hbox {\qquad (by \ref{lem:ker})}
\end{align*}
proving that
$$
\Phi_2( \{\bar{x},  \, \bar{y}, \, \bar{z}\}_1)= \{\Phi_2(\bar{x}),  \, \Phi_1(\bar{y}), \, \Phi_2(\bar{z})\}_1.
$$

\noindent -- Let us check that $\Phi_2( \{\bar{x},  \, \bar{y}, \, \bar{z}\}_2)= \{\Phi_2(\bar{x}),  \, \Phi_1(\bar{y}), \, \Phi_2(\bar{z})\}_2$:
\begin{align*}
\bullet\ &\Phi_2( \{\bar{x},  \, \bar{y}, \, \bar{z}\}_2)=-\frac 14({[[y,[z,a]],[x,a]]} +{[[y,[x,a]],[z,a]]} - {[y,[[x,[z,a]],a]]})+\text{Ker}_LB.\\
\bullet\ &\{\Phi_2(\bar{x}),  \, \Phi_1(\bar{y}), \, \Phi_2(\bar{z})\}_2
=\{-x+ \text{Ker}_LB, \, \frac14[[y,a],a] , \, -z+\text{Ker}_LB\}_2\\&
=-\frac14[[[[y,a],a],x],z]+\text{Ker}_LB.
\end{align*}
Let us see that
$$
[[[[y,a],a],x],z]-[[y,[z,a]],[x,a]]-[[y,[x,a]],[z,a]]+[y,[[x,[z,a]],a]]\in{\rm ker}_L\, \{a\}.
$$
In terms of operators acting on $y$,
\begin{align*}
   \bullet\  &[[[[[[y,a],a],x],z],a],a]=(y)A^2XZA^2=(y)A^2ZXA^2\\
   \bullet\  &-[[[[y,[z,a]],[x,a]],a],a]=-(y)AZAXA^2\\
   \bullet\  &-[[[[y,[x,a]],[z,a]],a],a]=-(y)AXAZA^2\\
   \bullet\  &[[[y,[[x,[z,a]],a]],a],a]=(y)(AXAZA^2+AZAXA^2-A^2ZXA^2)
\end{align*}
proving that
   $$\Phi_2( \{\bar{x},  \, \bar{y}, \, \bar{z}\}_2)= \{\Phi_2(\bar{x}),  \, \Phi_1(\bar{y}), \, \Phi_2(\bar{z})\}_2$$
   because ${\rm ker}_L\, \{a\}\subseteq  \text{Ker}_LB$.
\end{proof}

{\bf Declaration of competing interest}. None declared.

\bibliographystyle{plain}
\bibliography{ref}

\end{document}